\documentclass[11pt]{ip-journal}
\usepackage{etex}
\usepackage{latexsym}
\usepackage{bbm}
\usepackage{amscd}
\usepackage{amsmath}
\usepackage{amsfonts}
\usepackage{amssymb}

\newtheorem{theorem}{\bf Theorem}[section]
\newtheorem{definition}[theorem]{\bf Definition}
\newtheorem{lemma}[theorem]{\bf Lemma}
\newtheorem{prop}[theorem]{\bf Proposition}

\newtheorem{remark}[theorem]{\bf Remark}

\renewcommand{\theequation}{\arabic{section}.\arabic{equation}}
\renewcommand{\thefigure}{\thesection.\arabic{figure}}
\renewcommand{\thetable}{\thesection.\arabic{table}}

\def\no{\noindent}

\makeatletter

\newcommand{\Rmnum}[1]{\expandafter\@slowromancap\romannumeral #1@}
\makeatother
\numberwithin{equation}{section}
\usepackage{enumitem}
\usepackage{tikz-cd}
\usepackage{indentfirst}
\begin{document}
\bibliographystyle{plain}
\begin{center}{\LARGE \bf Area-minimizing Cones over Grassmannian Manifolds}
\end{center}

\begin{center}
Xiaoxiang Jiao and Hongbin Cui$^\ast$

School of Mathematical Sciences, University of Chinese Academy of Sciences, Beijing 100049, China,

$^\ast$Corresponding Author

E-mail: xxjiao@ucas.ac.cn; \phantom{,,}cuihongbin16@mails.ucas.ac.cn

\end{center}

\bigskip

\no
{\bf Abstract.}
It is a well-known fact that there exists a standard minimal embedding map for the Grassmannians of $n$-planes $G(n,m;\mathbb{F})(\mathbb{F}=\mathbb{R},\mathbb{C},\mathbb{H})$ and Cayley plane $\mathbb{O}P^2$ into Euclidean spheres, then an natural question is that if the cones over these embedded Grassmannians are area-minimizing?

In this paper, detailed descriptions for this embedding map are given from the point view of Hermitian orthogonal projectors which can be seen as an direct generalization of Gary R. Lawlor's(\cite{lawlor1991sufficient}) original considerations for the case of real projective spaces, then we re-prove the area-minimization of those cones which was gradually obtained in \cite{kerckhove1994isolated}, \cite{kanno2002area} and \cite{ohno2015area} from the perspectives of isolated orbits of adjoint actions or canonical embedding of symmetric $R$-spaces, all based on the method of Gary R. Lawlor's Curvature Criterion.

Additionally, area-minimizing cones over almost all common Grassmannians has been given by Takahiro Kanno, except those cones over oriented real Grassmannians $\widetilde{G}(n,m;\mathbb{R})$ which are not Grassmannians of oriented $2$-planes. The second part of this paper is devoted to complement this result, a natural and key observation is that the oriented real Grassmannians can be considered as unit simple vectors in the exterior vector spaces, we prove that all their cones are area-minimizing except $\widetilde{G}(2,4;\mathbb{R})$.\\

\no

\no

\bigskip

\tableofcontents
\section{Introduction}
\vspace{0.5cm}

$\textbf{Backgrounds \ on \ area-minimizing \ cones.}$

The classical Plateau's problem concerns on the existence of an area-minimizing surface bounded by a given Jordan curve in Euclidean space. This problem has a long and rich history and it admits many variations. In the development for high dimensional cases, geometry measure theory provides effective research method for it, there the considered surfaces are often refer to currents-an type of generalized surfaces, and a compact $k$-dimensional surface with boundary is area-minimizing in $\mathbb{R}^{n}$ if no other integral currents with the same boundary has less surface area.

Area-minimizing cones are a class of area-minimizing surfaces whose its truncated part inside the unit ball owns the least area among all integral currents with the same boundary. These cones are of particular importance for the existence of an isolated singularity at the origin which provide new examples of the solutions of non-interior regularity to Plateau's problem.

In their fundamental paper \textbf{Calibrated geometries} in 1982 \cite{harvey1982calibrated}, Harvey and Lawson have carried out an extensive research on the geometries determined by a differential form, which was called \textbf{calibration}. Generally, a calibration $\varphi$ on a Riemannian manifold $X$ is a closed, smooth, differential $p$-form of comass $1$ which means $\varphi$ attains maximum $1$ when restricted on all the Grassmannian of oriented $p$-plane of $T_{x}X$ where $x$ ranges over $X$. An amazing fact is that any calibrated surface is area-minimizing in its homology class after a simple application of Stokes' theorem.

The method of calibration can be used to determine the homologically area-minimizing circle on a given Riemannian manifold. For the Grassmannians, the research targets which we focus on this paper and also as a question posed in \cite{harvey1982calibrated}, such abundant results are presented in \cite{gluck1989calibrated},\cite{gluck1995volume}, \cite{gu1998stable},\cite{dadok1999pontryagin},\cite{grossman2001uniqueness}, etc.

If the ambient space of a surface $M$ is Euclidean space $\mathbb{R}^{n}$, its homology class just consists of those surfaces which have the same boundary with $M$, then a calibrated surface(cone) in $\mathbb{R}^{n}$ is (absolutely) area-minimizing. These results also hold with the smooth calibration replaced by its weak forms(like coflat calibrations, exterior derivatives of Lipschitz forms, etc, see \cite{harvey1982calibrated},
\cite{cheng1988area},\cite{lawlor1991sufficient},\cite{morgan1987calibrations},\cite{murdoch1991twisted}).
Several calibrations like normalized powers of K\"{a}hler forms, special Lagrangian forms, Cayley calibration, associative and coassociative calibrations are given by Harvey and Lawson, meanwhile certain area-minimizing cones are exhibited based on the calibrated theory: like those homogeneous hypercones which was introduced in \cite{lawson1972equivariant}(be calibrated by a class of coflat calibrations), the twisted normal cones associated with the compact austere submanifolds of sphere, and the coassociative minimal cone which was first constructed in \cite{lawson1977non} as an non-parametric Lipschitz solution to the minimal surface system of high codimension, etc.

Area-minimizing cones couldn't always be calibrated by some smooth differential forms, with an effort for searching necessary and sufficient conditions for a cone to be area-minimizing, Gary R. Lawlor(\cite{lawlor1991sufficient}) developed a general method for proving that a cone is area-minimizing, which was called Curvature Criterion by himself.

Lawlor explained his Curvature Criterion by two equivalent concrete objects, the vanishing calibration and the area-nonincreasing retraction, both defined on certain angular neighborhood of the cone rather than in the entire Euclidean space. They are linked by the fact that the tangent space of retraction surface is just the orthogonal complement of the dual of the vanishing calibration. They both derive an ordinary differential equation($ODE$). The area-minimizing tests include if the $ODE$ has solutions, what is the maximal existence interval of a solution, and then compare it with an important potential of the cone-\textbf{normal radius}. Lawlor also studied those cones which his Curvature Criterion is both necessary and sufficient like the minimal, isoparametric hypercones and the cones over principal orbits of polar group actions. Based on his method, the classification of minimal(area-minimizing, stable, unstable) cones over products of spheres is completed, he also proved some cones over unorientable real projective spaces, compact matrix groups are area-minimizing and gave new proofs of a large class of homogeneous hypercones being area-minimizing.

Other related researches, or researches for the area-minimizing cones associated to Lawlor's Curvature Criterion are: a new twistor calibrated theory and applications for the Veronese cone given by Tim Murdoch\cite{murdoch1991twisted}, also see an new point view of Veronese cone from vanishing calibration\cite{lawlor1995note}; Extending the definition of coflat calibration and the illustration of special Largraigian calibrated cones over compact matrix group which are also shown in \cite{lawlor1991sufficient} are given by Benny N. Cheng\cite{cheng1988area}; Proofs for area-minimization of $\textbf{Lawson-}$
$\textbf{Osserman}$ cones are given by Xiaowei Xu, Ling Yang and Yongsheng Zhang\cite{xu2018new}; Researches on area-minimizing cones associated with isoparametric foliations has been pioneered by Zizhou Tang and Yongsheng Zhang\cite{tang2020minimizing}, etc.

\vspace{0.5cm}
$\textbf{Backgrounds\ on \ area-minimizing \ cones \ associated \ to}$

$\textbf{Grassmannians.}$

The Grassmannian manifolds are important symmetric spaces which can be endowed with some special geometric structures. Generally, the research objects include the Grassmannian of $n$-plane in $\mathbb{F}^{m}$: $G(n,m;\mathbb{F})$
($\mathbb{F}=\mathbb{R},\mathbb{C},\mathbb{H}$), the Grassmannian of oriented $n$-plane in $\mathbb{R}^{m}$: $\widetilde{G}(n,m;\mathbb{R})$, and there exists an only analogue for $\mathbb{F}=\mathbb{O}$, the Cayley plane: $\mathbb{O}P^2$.

There exists standard embedding maps for those Grassmannian manifolds, except those $\widetilde{G}(n,m;\mathbb{R})(n,m-n\neq 2)$, into Euclidean space which can be seen as isolated orbits of some polar group actions, or standard embedding of symmetric $R$-spaces(\cite{hirohashi2000area},\cite{berndt2016submanifolds}). By considering $G(n,m;\mathbb{R})$ and $G(n,m;\mathbb{C})$ as isolated singular orbits of adjoint actions of special orthogonal groups and special unitary groups(\cite{kerckhove1994isolated}), Michael Kerckhove proved almost all their cones are area-minimizing. From the perspective of symmetric $R$-spaces and their canonical embeddings for $\widetilde{G}(2,2l+1;\mathbb{R})(l\geq 3)$, by constructing area-nonincreasing retractions directly, Daigo Hirohashi, Takahiro Kanno and Hiroyuki Tasaki(\cite{hirohashi2000area}) proved the area-minimization of these cones.

Later on, also from this point view, Takahiro Kanno(\cite{kanno2002area}) proved all the cones over the Grassmannian of subspaces $G(n,m;\mathbb{F})$( where $\mathbb{F}=\mathbb{R},\mathbb{C},\mathbb{H})$ and Grassmannian of oriented $2$-planes $\widetilde{G}(2,m;\mathbb{R})(m\geq 5)$ are area-minimizing. Moreover, in the language of Grassmannian, he also proved the area-minimization of cones over Lagrangian Grassmannian $U(n)/O(n)$-the set of all Lagrangian $n$-planes in $\mathbb{R}^{2n}$, complex Lagrangian Grassmannian $Sp(n)/U(n)$-the set of all complex Lagrangian $n$-planes in $\mathbb{C}^{2n}$, $SO(2n)/U(n)$-the Grassmannian of orthogonal complex structures on $\mathbb{R}^{2n}$, $U(2n)/Sp(n)$-the Grassmannian of quaternionic structures on $\mathbb{C}^{2n}$, etc.

In $2015$, the cone over $\mathbb{O}P^2$ was also shown area-minimizing from the above point view by Shinji Ohno and Takashi Sakai(\cite{ohno2015area}) as part of their results. Moreover, they proved that those cones over minimal products of R-spaces which each R-space associate to an minimizing cone(except $\mathbb{R}P^2$) are also area-minimizing.

We also note here, standard embeddings of $\mathbb{F}P^2(\mathbb{F}=\mathbb{R},\mathbb{C},\mathbb{H},\mathbb{O})$ into unit spheres of Euclidean spaces are often called Veronese embeddings, they are focal submanifolds of isoparametric foliations with the number of principal curvatures $g=3$ and of multiplicity $m_{1}=m_{2}=m_{3}=1,2,4,8$. In \cite{tang2020minimizing}, as part of their results, from the point view of isoparametric theory, Zizhou Tang and Yongsheng Zhang also exhibited the area-minimization of cones over $\mathbb{F}P^2(\mathbb{F}=\mathbb{C},\mathbb{H},\mathbb{O})$ and their minimal products.

Recently, also from the point view of Hermitian orthogonal projectors, in \cite{jiao2021area}, we make a further step on showing that the cone over minimal products of Grassmannians is area-minimizing if the dimension of cone is no less than eight, moreover, some new area-minimizing cones of dimension seven are also exhibited.

\vspace{1cm}

We summarize those results for the common Grassmannians in the following table:
\begin{center}
\begin{tabular}{|c|c|c|c|}
\hline
$\textrm{Submanifolds}$ & $\textrm{Ambient \ spaces}$ & $\textrm{Originally \ from}$ & $\textrm{Minimizing}$  \\ \hline
$G(n,m;\mathbb{R})$ & $\{X\in H(m;\mathbb{R})|tr X=0\}$ & $\textrm{Kerckhove(94)}$ & $\textrm{Except} \ \mathbb{R}P^2$ \\ \hline
$G(n,m;\mathbb{C})$ & $\{X\in H(m;\mathbb{C})|tr X=0\}$ & $\textrm{Kerckhove(94)}$ & $\textrm{All}$ \\ \hline
$G(n,m;\mathbb{H})$ & $\{X\in H(m;\mathbb{H})|tr X=0\}$ & $\textrm{Kanno(02)}$ & $\textrm{All}$ \\ \hline
$\mathbb{O}P^2$ & $\{X\in H(3;\mathbb{O})|tr X=0\}$ & $\textrm{Ohno,etc(15)}$ & $\textrm{Yes}$ \\ \hline
$\widetilde{G}(2,2l+1;\mathbb{R})$ & $\mathfrak{so}(2l+1)$ & $\textrm{Hirohashi,etc(00)}$ & $l\geq 3$ \\ \hline
$\widetilde{G}(2,2l;\mathbb{R})$ & $\mathfrak{so}(2l)$ & $\textrm{Kanno(02)}$ & $l\geq 4$ \\ \hline
\end{tabular}
\end{center}
where the ambient space $H(m;\mathbb{F})(\mathbb{F}=\mathbb{R},\mathbb{C},\mathbb{H})$ denote the Euclidean space consists of all Hermitian matrices over $\mathbb{F}$, and the Lie algebra $\mathfrak{so}(2l+1),\mathfrak{so}(2l)$ denote the Euclidean space consists of all skew-symmetric matrices over $\mathbb{R}$. We note here Takahiro Kanno(\cite{kanno2002area}) actually have exhibited the area-minimization for the cones over the left $\widetilde{G}(2,5;\mathbb{R})$ and $\widetilde{G}(2,6;\mathbb{R})$, then:

\vspace{0.5cm}
$\mathbf{Question.}$ Similar to $\widetilde{G}(2,m;\mathbb{R})$, what is the area-minimizing cones over $\widetilde{G}(n,m;\mathbb{R})$ ?

\vspace{0.3cm}
we will give an feasible answer for this question in the final chapter.

\vspace{0.5cm}
$\mathbf{Statement \ of \ Results}$
In chapter $3,4,5$, we re-prove the following results by regarding  $G(n,m;\mathbb{F})$$(\mathbb{F}=\mathbb{R},\mathbb{C},\mathbb{H})$, $\mathbb{O}P^2$ as Hermitian orthogonal projectors uniformly.

\begin{theorem}
The cones over nonoriented real Grassmannian
$G(n,m;\mathbb{R})$, complex Grassmannian $G(n,m;\mathbb{C})$, quaternion Grassmannian $G(n,m;\mathbb{H})$ are area-minimizing, where $m\geq 2n \geq 4$.
\end{theorem}

\begin{theorem}
The cones over $\mathbb{C}P^{m-1}$,$\mathbb{H}P^{m-1}$ are area-minimizing, where $m\geq 2$.
\end{theorem}

\begin{theorem}
The cone over $\mathbb{O}P^2$  is area-minimizing.
\end{theorem}

Additionally, for the left $\widetilde{G}(n,m;\mathbb{R})$-the oriented real Grassmannian manifolds
(we omit all oriented projective spaces over $\mathbb{R}$ which can be identified with the spheres), motivated by the classical Pl\"{u}cker embedding of Grasssmannian in projective space, we also call the natural map Pl\"{u}cker  embedding which considering  $\widetilde{G}(n,m;\mathbb{R})$ as the collection of unit simple vectors in the exterior vector space, in chapter $6$, we establish

\begin{theorem}
Except $\widetilde{G}(2,4;\mathbb{R})$, all the cones over Pl\"{u}cker \ embedding of $\widetilde{G}(n,m;\mathbb{R})$ are area-minimizing.
\end{theorem}

This paper is organized as follows:

In chapter $2$, we review the knowledge of standard embeddings of Grassmannian manifolds into Euclidean spaces as Hermitian orthogonal projectors, for later citations, we add an proof for the equivalent minimality of an immersed submanifold in sphere and its cone based on the method of moving frame(\cite{Chern1983}), especially the quantitative relations between their second fundamental forms which were used in \cite{lawlor1991sufficient} on many occasions.

In chapter $3$, we exhibit that $G(n,m;\mathbb{F})(\mathbb{F}=\mathbb{R},\mathbb{C},\mathbb{H})$ and $G(m-n,m;\mathbb{F})$ can be embedded into the same one Euclidean sphere as a pair of opposite minimal submanifolds simultaneously, the associated cones are opposite, and these cones are proved area-minimizing by Lawlor's Curvature Criterion. As special cases, the proof for area-minimization of cones over projective spaces are given separated in chapter $4$.

In chapter $5$, based on \cite{harvey1990spinors}, also considering the Cayley plane $\mathbb{O}P^2$ as the set of Hermitian orthogonal projection operators, we prove that this cone is area-minimizing.

In chapter $6$, we focus on the cones over oriented real Grassmannians $\widetilde{G}(n,m;\mathbb{R})$, the known area-minimizing cones are those cones over canonical embeddings of $\widetilde{G}(2,m;\mathbb{R})(m\geq 5)$ into $\mathfrak{so}(m)$ as symmetric $R$-spaces given in \cite{hirohashi2000area} and \cite{kanno2002area}. Notice that the ambient spaces are the Euclidean spaces $\mathfrak{so}(m)$ which can be identified with the exterior vector space $\wedge^{2}\mathbb{R}^{m}$. Similarly, we extend cones over $\widetilde{G}(2,m;\mathbb{R})$ to cones over the general oriented Grassmannians $\widetilde{G}(n,m;\mathbb{R})$ by embedding them in $\wedge^{n}\mathbb{R}^{m}$ (\cite{harvey1982calibrated},\cite{morgan1985exterior},\cite{gluck1995volume})-the so called Pl\"{u}cker embedding. We note here that the second fundamental forms and some related geometries associated with this embedding map have been studied by Wei Huan Chen in\cite{chen1988the}, meanwhile, we prove that the normal radius of these cones are all equal $\frac{\pi}{2}$. The final conclusion is that, the cone over $\widetilde{G}(2,4;\mathbb{R})$ is unstable, except this, all the others are area-minimizing.

\section{Preliminaries}
\subsection{Standard embedding of Grassmannian into Euclidean space}
The embedding of Grassmannian into Euclidean space as Hermitian orthogonal projectors is an well-known fact(\cite{kobayashi1968isometric}), for projective spaces and Cayley plane, these embeddings are just their first standard minimal immersions into unit spheres associated to the first eigenvalues of Laplace operator(\cite{chen1984total}).

Throughout this paper, $\mathbb{F}$ will denote the field of real numbers $\mathbb{R}$, the field of complex numbers $\mathbb{C}$, the normed quaternion associative algebra $\mathbb{H}$ unless otherwise stated, this subchapter is mainly based on \cite{chen1984total}, where the author had given detailed researches for the cases of projective spaces.

We denote \begin{equation}\notag
 d=dim_{R}\mathbb{F}=
 \begin{cases}
 1 & \text{if} \ \mathbb{F}=\mathbb{R},\\
 2 & \text{if} \ \mathbb{F}=\mathbb{C},\\
 4 & \text{if} \ \mathbb{F}=\mathbb{H}.\\
 \end{cases}
 \end{equation}

Let $\mathbb{F}^{m}$ be the right Hermitian vector space of column $m$-vectors with the inner product
\begin{equation}\notag
\langle z,w\rangle_{\mathbb{F}}=z^{*}w=\sum_{i=1}^{m}\bar {z}^{i} w^{i},
\end{equation}
where $z=(z^{i}),w=(w^{i})\in \mathbb{F}^{m}$, $\bar {z}^{i}$ is the standard conjugation  and $z^{*}=\bar{z}^{t}$ denote the conjugate transpose.

The Grassmannian manifold $G(n,m;\mathbb{F})$ is the set of all $n$-dimensional subspaces in $\mathbb{F}^{m}$, when $n=1$, it is just the projective space. Denote $\widetilde{G}(n,m;\mathbb{R})$ be the set of oriented $n$-planes in $\mathbb{R}^n$, it is the double cover of $G(n,m;\mathbb{R})$.

We use the following notations: $M(m;\mathbb{F})$ denote the space of all $m\times m$ matrices over $\mathbb{F}$, $H(m;\mathbb{F})\subset M(m;\mathbb{F})$ denote the space of all Hermitian matrices over $\mathbb{F}$,
\begin{equation}\notag
H(m;\mathbb{F})=\{A\in M(m;\mathbb{F})|A^{*}=A\}.
\end{equation}

Let $U(m;\mathbb{F})=\{A\in M(m;F)|AA^{*}=A^{*}A=I\}$ denote the $\mathbb{F}$-unitary group(For $\mathbb{F}=\mathbb{H}$, we will quickly see that $A$ also own two-sided inverse by considering the complex representation of $ M(m;\mathbb{H})$,\cite{zhang1997quaternions}), it preserve the Hermitian inner product of $\mathbb{F}^{m}$. Moreover, $U(m;\mathbb{R})=O(m)$, $U(m;\mathbb{C})=U(m)$, $U(m;\mathbb{H})=Sp(m)$.

$H(m;\mathbb{F})$ can be identified with real Euclidean space $E^{N}$ with the inner product: \begin{equation}\notag
g(A,B)=\frac{1}{2} \mathrm{Re}\ tr_{\mathbb{F}}(AB),
\end{equation}
where $A,B\in H(m;\mathbb{F})$,  $N=m+dm(m-1)/2$.

$U(m;\mathbb{F})$ has an adjoint action $\rho$ on $H(m;\mathbb{F})$ given by
\begin{equation}\notag
\begin{aligned}
\rho :U(m;\mathbb{F})\times H(m;\mathbb{F})&\rightarrow H(m;\mathbb{F})\\
 (Q,P) &\mapsto QPQ^{*},
\end{aligned}
\end{equation}
where $Q\in U(m;\mathbb{F}), P\in H(m;\mathbb{F})$.

\begin{prop}
$\rho^{*}g=g$ on $H(m;\mathbb{F})$.
\end{prop}

\begin{proof}
The cases $\mathbb{F}=\mathbb{R}$, $\mathbb{C}$ are trivial and $\mathrm{Re}$ is not needed in the definition of $g$. For $\mathbb{F}=\mathbb{H}$, choose $U\in Sp(m)$, then we need to verify: $\mathrm{Re}\ tr_{\mathbb{F}}(UABU^{*})= \mathrm{Re}\ tr_{\mathbb{F}}(AB)$, the left hand is $\sum_{i,j,k}\mathrm{Re} \ U_{ij}(AB)_{jk}U^{*}_{ki}$, since $\mathrm{Re} \ (pq)=\mathrm{Re} \ (qp)$, then
\begin{equation}\notag
\sum_{i,j,k}\mathrm{Re} \ U_{ij}(AB)_{jk}U^{*}_{ki}=\sum_{i,j,k}\mathrm{Re} \ U^{*}_{ki}U_{ij}(AB)_{jk}=\mathrm{Re} \ tr_{\mathbb{F}}(AB).
\end{equation}
\end{proof}

For $L\in G(n,m;\mathbb{F})$ and $L^{\perp}\in G(m-n,m;\mathbb{F})$, choose an $\mathbb{F}$-unitary basis $\{z_{1},\ldots,z_{m}\}$ of $\mathbb{F}^{m}$ such that $L=span_{\mathbb{F}}\{z_{1},\ldots,z_{n}\}$ and $L^{\perp}=span_{\mathbb{F}}\{z_{n+1},\ldots,z_{m}\}$, any other orthonormal basis of $L$ are given by $\{(z_{1},\ldots,z_{n})A|A\in U(n;\mathbb{F})\}$. We denote $P_{L}$ the operator of $\mathbb{F}$-Hermitian orthogonal projection onto the $n$-plane $L\in G(n,m;\mathbb{F})$, let $f$ be the matrix $(z_{1},\ldots,z_{n})$, $g$ be the matrix $(z_{n+1},\ldots,z_{m})$, then for any vector $w\in \mathbb{F}^{m}$,
\begin{equation}\notag
P_{L}(w)=\sum_{i=1}^{n}z_{i}\langle z_{i},w\rangle_{\mathbb{F}}=(\sum_{i=1}^{n}z_{i}z_{i}^{*})w=ff^{*}w,
\end{equation}
and
\begin{equation}\notag
P_{L^{\perp}}(w)=\sum_{\alpha=n+1}^{m}z_{\alpha}\langle z_{\alpha},w\rangle_{\mathbb{F}}=(\sum_{\alpha=n+1}^{m}z_{\alpha}z_{\alpha}^{*})w=gg^{*}w,
\end{equation}
easy to see $P_{L}+P_{L^{\perp}}=ff^{*}+gg^{*}=I$, $P_{L}(v)=0$ if and only if $v\in L^{\perp}$.

For an $\mathbb{F}$-unitary basis $f=(z_{1},\dots,z_{n})$ of $L$, we define $\tilde{\varphi} (f)=\sum_{i=1}^{n}z_{i}z_{i}^{*}=ff^{*}$, since for every $i\in\{1,\ldots,n\}$, $tr_{\mathbb{F}}(z_{i}z_{i}^{*})=|z_{i}|^{2}=1$, then we see $ff^{*}$ is of $\mathbb{F}$-trace $n$. Let $\tilde{f}=fA$ be another $\mathbb{F}$-unitary basis for $L$, then $\tilde{\varphi} (\tilde{f})=\tilde{f}\tilde{f}^{*}=fAA^{*}f^{*}=ff^{*}$ is still of $\mathbb{F}$-trace $n$, so $\tilde{\varphi}$ descends to an embedding map:
\begin{equation}\notag
\begin{aligned}
\varphi:G(n,m;\mathbb{F})&\rightarrow H(m;\mathbb{F})\\
 L &\mapsto P_{L},
\end{aligned}
\end{equation}
and $P_{L}^{2}=P_{L}$, $L=\{z\in\mathbb{F}^{m}|P_{L}z=z\}$, we note here though the trace of an quaternion linear endomorphism of $\mathbb{H}^{m}$ is not well defined(\cite{zhang1997quaternions}), but for the Hermitian orthogonal projective operator in the above, it does be.

\begin{prop}
Hence we get an isomorphism:
\begin{equation}\notag
G(n,m;\mathbb{F})\cong \varphi(G(n,m;\mathbb{F}))=\{P\in H(m;\mathbb{F})|P^{2}=P, tr_{\mathbb{F}}P=n\}.
\end{equation}
\end{prop}

\begin{proof}
In the proof, we need some knowledge about quaternion linear algebra, a good reference is \cite{rodman2014topics}.
Follow the above discussions, we only need to prove the converse part, which is that for every element of $\{P\in H(m;\mathbb{F})|P^{2}=P, tr_{\mathbb{F}}P=n\} $, there exists an $L\in G(n,m;\mathbb{F})$ associated with it, define $S=\{z\in\mathbb{F}^{m}|Pz=z\}$, we want  to show $dim_{\mathbb{F}}S=n$.  Note for every $P\in H(m;\mathbb{F})$ and $P^2=P$, there exists an unitary matrix $U\in U(m;\mathbb{F})$, such that
\begin{equation}\notag
U^{*}PU=\Lambda=
\begin{pmatrix}
I_{r} & 0\\
0 &0
\end{pmatrix}
\end{equation}
for an nonnegative number $r=0,\ldots,n$, where in case of $\mathbb{F}=\mathbb{H}$, see Theorem 5.3.6 in \cite{rodman2014topics} by multiplying a series of elementary matrix both in the left and right.

For $S=\{z\in\mathbb{F}^{m}|Pz=z\}$, denote $z=Uw$, it is $\mathbb{F}$-linear isomorphism since $U$ is nonsingular, then $S=\{w\in\mathbb{F}^{m}|\Lambda w=w\}$, hence $dim_{\mathbb{F}}S=r$.

To show $r=n$, we need to verify that the $tr_{\mathbb{F}}$ function of an Hermitian matrix(note all its diagonal elements are real numbers) is invariant under unitary adjoint action $\rho$, this is trivial for $\mathbb{F}=\mathbb{R},\mathbb{C}$.

For $P\in H(m;\mathbb{H})$, we consider its real matrix representation $\chi$ which satisfy the properties: $\chi(PT)=\chi(P)\chi(T)$ and $\chi(P^{*})=(\chi(P))^{t}$ for any $P,T\in M(m;\mathbb{H})$(\cite{rodman2014topics}). Then
\begin{equation}\notag
tr_{\mathbb{H}}P=n \Leftrightarrow tr\chi(P)=4n \Leftrightarrow tr\chi(\Lambda)=tr\chi(U^{*}PU)=tr \chi(P)=4n,
\end{equation}
shows that $r=n$.
\end{proof}

\begin{remark}
An alternative form $G(n,m;\mathbb{H})=\{P\in H(m;\mathbb{H})|P^{2}=P, tr_{\mathbb{R}}P=4n\}$ is given in \cite{harvey1990spinors}.
\end{remark}

From the above proposition, $\rho$ is an isometric, transitive action when restricted on $\varphi(G(n,m;\mathbb{F}))$, let $P_{0}=\begin{pmatrix}I_{n} & 0 \\ 0 & 0\end{pmatrix}$ be the origin of $G(n,m;\mathbb{F})$, $P_{0}^{\perp}=\begin{pmatrix} 0 & 0 \\ 0 & I_{m-n}\end{pmatrix}$ be the origin of $G(m-n,m;\mathbb{F})$, then $G(n,m;\mathbb{F})$ is the $\rho$-orbit through $P_{0}$, $G(m-n,m;\mathbb{F})$ is the $\rho$-orbit through $P_{0}^{\perp}$.

In the following, we will always identify $G(n,m;\mathbb{F})$ with $\varphi(G(n,m;\mathbb{F}))$.

Additionally,
\begin{prop}
denote $P_{L}^{\perp}=P_{L^{\perp}}$, then
\begin{equation}\label{fff}
P_{L}-P_{L}^{\perp}=2P_{L}-I\in U(m;\mathbb{F})
\end{equation}
is the totally geodesic Cartan embedding of the symmetric space into associated isometry group(\cite{cheeger2008comparison},\cite{uhlenbeck1989harmonic}).
\end{prop}
\begin{proof}
 Assume $A=\begin{pmatrix}A_{1}&A_{2}\\A_{3}&A_{4}\end{pmatrix}\in U(m;\mathbb{F})$ and $f=\begin{pmatrix}A_{1}\\A_{3}\end{pmatrix}$(respectively, $g=\begin{pmatrix}A_{2}\\A_{4}\end{pmatrix}$) is an $\mathbb{F}$-unitary basis of $L\in G(n,m;\mathbb{F})$(respectively, $L^{\perp}\in G(m-n,m;\mathbb{F})$), then $P_{L}=ff^{\ast}$ and $P_{L^{\perp}}=gg^{\ast}$, the unitary condition is
\begin{equation}\notag
ff^{\ast}+gg^{\ast}=I,
\end{equation}

the canonical involution is
\begin{equation}\notag
\begin{aligned}
\sigma :U(m;\mathbb{F})& \longrightarrow U(m;\mathbb{F})\\
A&\mapsto SAS,
\end{aligned}
\end{equation}
where
\begin{equation}\notag
S=\begin{pmatrix}I_{n}&0\\0&-I_{m-n}\end{pmatrix}
\end{equation}.

Then,
\begin{equation}\notag
  A\sigma(A^{-1})S=ASA^{\ast}=ff^{\ast}-gg^{\ast}=2P_{L}-I
\end{equation}

So the Cartan embedding in \eqref{fff} is equivalent to the definition in \cite{cheeger2008comparison} up to an constant right translation.
\end{proof}

\begin{remark}
The unitary transformation $I-2P_{L}$ is actually the reflection along the subspace $L$ through $L^{\perp}$(\cite{harvey1990spinors}).
\end{remark}

\subsection{Tangent and normal spaces of Grassmannian}
We choose a curve $\alpha (t)\in G(n,m;\mathbb{F})$, $\alpha (0)=P, \alpha^{'}(0)=X\in T_{P}G(n,m;\mathbb{F}) \subset H(m;\mathbb{F})$. Then derivative $\alpha (t)\alpha(t)=\alpha (t)$, we obtain $PX+XP=X$, i.e. $T_{P}G(n,m;\mathbb{F})\subset \{X\in H(m;\mathbb{F})|XP+PX=X\}$, they are all linear spaces over $\mathbb{R}$. Assume $P=QP_{0}Q^{*}$, then set $X=QX_{0}Q^{*}$, we get $X_{0}P_{0}+P_{0}X_{0}=X_{0}$. Next, we compute the dimension of the subspace $\{X_{0}\in H(m;\mathbb{F})|X_{0}P_{0}+P_{0}X_{0}=X_{0}\}$.

For any $X_{0}\in H(m;\mathbb{F})$, we put

\begin{equation}\notag
X_{0}=\begin{pmatrix}
		A & B\\B^{*} & D
		\end{pmatrix} ,A\in H(n;\mathbb{F}), D\in H(m-n;\mathbb{F}),
 \end{equation}
and  $B$  is  an arbitrary $n\times(m-n)$ matrix  over $\mathbb{F}$.

 Then \begin{equation}\notag
 X_{0}P_{0}+P_{0}X_{0}=X_{0}\Leftrightarrow X_{0}
 =\begin{pmatrix}
		0 & B\\B^{*} & 0
		\end{pmatrix},
\end{equation}
and the subspace $\{X_{0}\in H(m;\mathbb{F})|X_{0}P_{0}+P_{0}X_{0}=X_{0}\}$ has the same dimension $dn(m-n)$ with the tangent space at $P_{0}$, so does for every point of $G(n,m;\mathbb{F})$, hence:
\begin{equation}\notag
T_{P}G(n,m;\mathbb{F})=T_{P^{\perp}}G(m-n,m;\mathbb{F})=\{X \in H(m;\mathbb{F})|XP+PX=X\}.
\end{equation}

$G(n,m;\mathbb{F})$ is an orbit of $\rho$ through the origin $P_{0}$, the for another point $P=QP_{0}Q^{*}=Q\cdot P_{0}$, $Q\in U(m;\mathbb{F})$, let $X_{0}\in T_{P_{0}}G(n,m;\mathbb{F})$, then $X=Q\cdot X_{0}=QX_{0}Q^{*}\in T_{P}G(n,m;\mathbb{F})$ builds isomorphism between tangent spaces, it is just isotropy representation when $Q$ belongs to the isotropy subgroup at $P_{0}$.

Let $T_{P}H(m;\mathbb{F})=T_{P}G(n,m;\mathbb{F})\oplus N_{P}G(n,m;\mathbb{F})$. A vector $\xi$ is in the normal space $N_{P}G(n,m;\mathbb{F})$ if and only if $g(X,\xi)=0$ for all $X\in T_{P}G(n,m;\mathbb{F})$, for $P=QP_{0}Q^{*}$, we set $X=QX_{0}Q^{*}=Q\begin{pmatrix}
		0 & B\\B^{*} & 0
		\end{pmatrix}Q^{*}$,
$\xi=Q\xi_{0}Q^{*}=Q\begin{pmatrix}
		X & Y\\Y^{*} & Z
		\end{pmatrix}Q^{*}$, since $\rho$ is an isometry, then $g(X,\xi)=0$ is equivalent to $g(X_{0},\xi_{0})=0$, and $\mathrm{Re }\ tr(X_{0} \xi_{0})=0$ if and only if $Y=0$.

So, \begin{equation}\label{normal vectors at orgin}
N_{P_{0}}G(n,m;\mathbb{F})=\Biggl\{\begin{pmatrix}
		X & 0\\0 & Z
		\end{pmatrix}|X\in H(n;\mathbb{F}), Z\in H(m-n;\mathbb{F})\Biggr\}.
\end{equation}

On the other hand, $\xi P=P\xi$ iff $\xi_{0}P_{0}=P_{0}\xi_{0}$, and it is equivalent to $Y=0$.

Hence,
\begin{equation}\notag
N_{P}G(n,m;\mathbb{F})=N_{P^{\perp}}G(m-n,m;\mathbb{F})=\{\xi \in H(m;\mathbb{F})|\xi P=P \xi \}.
\end{equation}

\begin{lemma}\label{matrix identity}
Let $X,Y\in T_{P}G(n,m;\mathbb{F})$ be two tangent vectors at $P$, and $Z\in H(m;\mathbb{F})$ is an arbitrary vector, then:
(1)~$P,I$ and $XY\in N_{P}G(n,m;\mathbb{F})$;
(2)~$PZ+ZP-2PZP\in T_{P}G(n,m;\mathbb{F})$;
(3)~$PXP=0$.
\end{lemma}
\begin{proof}
$P,I$ naturally commute with $P$; Since $PX+XP=X,PY+YP=Y$, then $(PX+XP)Y=X(PY+YP)$ shows that $XY$ commute with $P$; easy to calculate $P(PZ+ZP-2PZP)+(PZ+ZP-2PZP)P=PZ+ZP-2PZP$; multiply $P$ in the right of $PX+XP=X$ shows that $PXP=0$.
\end{proof}

Hence, we have the following Gauss formula for the embedding of $G(n,m;\mathbb{F})$ in $H(m;\mathbb{F})$ which is similar to the cases of projective spaces
\cite{chen1984total}:
\begin{prop}\label{Normal forms}
Let $X,Y$ be two tangent vector fields of $G(n,m;\mathbb{F})$, then:
\begin{equation}
\begin{cases}
h(X,Y)=(XY+YX)(I-2P); \\
\nabla_{X}Y=2(XY+YX)P+P(\tilde{\nabla}_{X}Y)+(\tilde{\nabla}_{X}Y)P,
\end{cases}
\end{equation}
where $\widetilde{\nabla},\nabla$ are the Riemannian connections of $H(m;\mathbb{F}),G(n,m;\mathbb{F})$, $h$ denote the second fundamental form of $G(n,m;\mathbb{F})$ in $H(m;\mathbb{F})$.
\end{prop}
\begin{proof}
Let $P(t)$ be a curve along $G(n,m;\mathbb{F})$ in $H(m;\mathbb{F})$, consider $X=X_{P}\in T_{P}G(n,m;\mathbb{F})$, $P(t)$ satisfy: $P(0)=P,P'(0)=X$, also denote $Y=Y(t)$ be the restriction of the tangent vector field $Y(t)$ on $P(t)$, then
\begin{equation}\notag
Y(t)=P(t)Y(t)+Y(t)P(t)
\end{equation},
from which we find
\begin{equation}\notag
\widetilde{\nabla}_{X}Y=P(\widetilde{\nabla}_{X}Y)+(\widetilde{\nabla}_{X}Y)P+XY+YX,
\end{equation}
denote $f=(XY+YX)(I-2P),g=2(XY+YX)P+P(\widetilde{\nabla}_{X}Y)+(\widetilde{\nabla}_{X}Y)P$, then
\begin{equation}\notag
\widetilde{\nabla}_{X}Y=f+g.
\end{equation}

(3) of Lemma \ref{matrix identity} tells us: $P(t)Y(t)P(t)=0$, then we have $XYP+PYX+P(\widetilde{\nabla}_{X}Y)P=0$,
since $YXP=PYX$, then
\begin{equation}\notag
(XY+YX)P=-P(\widetilde{\nabla}_{X}Y)P,
\end{equation}

From (2) of Lemma \ref{matrix identity}, we find $g\in T_{P}G(n,m;\mathbb{F})$ and easy to check $f\in N_{P}G(n,m;\mathbb{F})$.
\end{proof}
\subsection{Relations between minimal submanifold in sphere and its cone}
Let $M$ be a $(p-1)$-dimensional, compact submanifold which is immersed into the unit sphere of $\mathbb{R}^{n}$, the immersion map is denoted by $X$, then the cone $C(M)$ is described by:
\begin{equation}\notag
\begin{aligned}
Y:M\times (0,\infty)&\rightarrow \mathbb{R}^{n} \\
(m,t) &\mapsto tm,
\end{aligned}
\end{equation}
i.e. $Y(m,t)=tX(m)$.

The truncated cone $C(M)_{\epsilon}:=C(M)|_{M\times [\epsilon,1]}(\epsilon >0)$.

$M$ and $C(M)$ are all immersed submanifolds in $\mathbb{R}^{n}$ with the canonical Euclidean metric, so we consider the following vectors being of the standard Euclidean lengths.

Locally, given the orthonormal tangent vector field $\{e_{i}\}(i=1,\ldots,p-1)$ and orthonormal  normal vector field $\{e_{\alpha}\}(\alpha=p, \ldots, n-1)$ of $M$ in $S^{n-1}(1)$, then $e_{i}(m,t):=e_{i}(m)$ and $e_{\alpha}(m,t):=e_{\alpha}(m)$ are orthonormal tangent vectors and orthonormal normal vectors of $C(M)$ at $(m,t)$, $m\in X(m)$.

We use the following notations for $M$: the dual frames of $M$ are $\omega^{i}$, the connection forms of tangent bundle are $\omega_{i}^{j}$, the connection forms of normal bundle are $\omega_{i}^{\alpha}$, the coefficients of second fundamental forms are $h^{\alpha}_{ij}$.

Then, restrict $\omega^{\alpha}$ on submanifold $M$:
\begin{equation}\notag
\omega^{\alpha}=0\Rightarrow \omega_{i}^{\alpha}=h^{\alpha}_{ij}\omega^{j},h^{\alpha}_{ij}=h^{\alpha}_{ji},
\end{equation}
$M$ is minimal if and only if $\sum_{i}h^{\alpha}_{ii}=0$, for every $\alpha=p,\ldots,n-1$.

For the cone $C(M)$, set $e_{0}(m)=X(m)$, from
\begin{equation}\notag
de_{0}=dX=\omega^{i}e_{i},
\end{equation}
we have $\omega_{0}^{i}=\omega^{i},\omega_{0}^{\alpha}=0$.

The tangent frames of $C(M)$ in $\mathbb{R}^{n}$ are given by $\{e_{0},e_{i}\}(i=1,\ldots,p-1)$, the normal frames of $C(M)$ in $\mathbb{R}^{n}$ are given by $\{e_{\alpha} \}(\alpha=p,\ldots,n-1)$. Let $\{\theta^{0},\theta^{i}\}$ be the dual frames, then from
\begin{equation}\notag
dY=dtX+tdY=dte_{0}+t\omega^{i}e_{i},
\end{equation}
we have
\begin{equation}\notag
\theta^{0}=dt,\theta^{i}=t\omega^{i}.
\end{equation}

The connection forms of tangent bundle of $C(M)$ are given by
\begin{equation}\notag
\begin{aligned}
\theta_{0}^{i}&=\langle de_{0},e_{i}\rangle=\omega^{i},\\
\theta_{j}^{i}&=\langle de_{j},e_{i}\rangle=\omega_{j}^{i}.
\end{aligned}
\end{equation}

And \begin{equation}\notag
\begin{aligned}
de_{\alpha}&=\theta_{\alpha}^{0}e_{0}+\theta_{\alpha}^{i}e_{i}+\theta_{\alpha}^{\beta}e_{\beta}\\
&=\omega_{\alpha}^{i}e_{i}+\omega_{\alpha}^{\beta}e_{\beta},
\end{aligned}
\end{equation}
so $\theta_{\alpha}^{0}=\theta_{0}^{\alpha}=0$, $\theta_{\alpha}^{i}=\omega_{\alpha}^{i},\theta_{\alpha}^{\beta}=\omega_{\alpha}^{\beta}$.

Finally, let $\tilde{h}_{sl}^{\alpha}(s,l=0,1,\ldots,p-1)$ be the coefficients of second fundamental forms of $C(M)$, from
\begin{equation}\notag
\theta_{i}^{\alpha}=\omega_{i}^{\alpha}=\sum_{j=1}^{p-1}h_{ij}^{\alpha}\omega^{j}
=\sum_{l=0}^{p-1}\tilde{h}_{il}^{\alpha}\theta^{l},
\end{equation}
we conclude that
\begin{equation}\label{mim submfd and its cone}
\tilde{h}_{i0}^{\alpha}=0,\tilde{h}_{00}^{\alpha}=0,\tilde{h}_{ij}^{\alpha}(tm)=\frac{1}{t}h_{ij}^{\alpha}(m),
\end{equation}
an intuitive explanation is that at infinity, the cone behave more flat.

Let the second fundamental form of $M$ in $S^{n-1}(1)$(respectively, $C(M)$ in $\mathbb{R}^{n}$) along the normal direction $e_{\alpha}$ by $B^{\alpha}$(respectively, $\tilde{B}^{\alpha}$), then
\begin{equation}\notag
tr\tilde{B}^{\alpha}=\tilde{h}_{00}^{\alpha}+\sum_{i=1}^{p-1}\tilde{h}_{ii}^{\alpha}=\frac{1}{t} \sum_{i=1}^{p-1}h_{ii}^{\alpha}=\frac{1}{t}tr B^{\alpha}.
\end{equation}

Hence,
\begin{prop}
$M$ is minimal in $S^{n-1}(1)$ if and only if its cone $C(M)$ is minimal in $\mathbb{R}^{n}$.
\end{prop}

Let the set of eigenvalues of $B^{\alpha}$(or the shape operator $A_{\alpha}$) at $m$ be $\{\lambda_{1},\ldots,\lambda_{p-1}\}$, $\tilde{B}^{\alpha}$ at $tm$ has eigenvalues $\{0,\frac{1}{t}\lambda_{1},\ldots,\frac{1}{t}\lambda_{p-1}\}$, then

\begin{prop}(\cite{harvey1982calibrated})
$M$ is an austere submanifold of $S^{n-1}(1)$ if and only if its cone is an austere submanifold of $\mathbb{R}^{n}$.
\end{prop}

\section{Cones over the general Grassmannians}
We will adopt the following ranges of indices in this section unless otherwise stated:
\begin{center}
$1\leq a,b,c,d \leq n$;

$n+1\leq \alpha,\beta,\lambda,\mu \leq m$;

$u,v,w,z \in \{1, i, j, k\}$,
\end{center}
where $i,j,k$ denote the imaginary units of $\mathbb{F}=\mathbb{H}$, $i$ denote the imaginary unit of $\mathbb{F}=\mathbb{C}$.
\subsection{Minimal embedding of Grassmannian}
Let $e_{a}$ be the column vectors with $1$ in the $a^{th}$ slot, denote $E_{ab}=e_{a}e_{b}^{T}$, a simple chain rule is $E_{ab}E_{cd}=e_{a}e_{b}^{T}e_{c}e_{d}^{T}=\delta_{bc}E_{ad}$. Now let $F_{\alpha a}^{u}$ denote the matrix $uE_{\alpha a}+\bar {u} E_{a\alpha}$, where $\bar {u}$ denote the conjugation on $\mathbb{F}$. Then it is easy to verify that $F_{\alpha a}^{u}$ is an orthonormal basis of $T_{P_{0}}G(n,m;\mathbb{F})$.

Now $F_{\alpha a}^{u}F_{\beta b}^{v}+F_{\beta b}^{v}F_{\alpha a}^{u}=\delta_{\alpha \beta}(\bar u vE_{ab}+\bar v uE_{ba})+\delta_{a b}(u \bar vE_{\alpha \beta}+v \bar u E_{\beta \alpha})$, then the second fundamental forms are given by:
\begin{equation}\notag
h(F_{\alpha a}^{u},F_{\beta b}^{v})=-\delta_{\alpha \beta}(\bar u vE_{ab}+\bar v uE_{ba})+\delta_{a b}(u \bar vE_{\alpha \beta}+v \bar u E_{\beta \alpha}).
\end{equation}

So, $h(F_{\alpha a}^{u},F_{\alpha a}^{u})=2(E_{\alpha \alpha}-E_{a a})$, the mean curvature vector field at $P_{0}$ is:
\begin{equation}\notag
\mathbf{H}_{P_{0}}=\frac{2}{dn(m-n)}
\begin{pmatrix}
-(m-n)I & 0\\ 0 & nI
\end{pmatrix}=-\frac{2m}{dn(m-n)}\left(P_{0}-\frac{n}{m}I\right).
\end{equation}

Since $G(n,m;\mathbb{F})$ are orbits of isometry group actions and $g(P-\frac{n}{m}I,P-\frac{n}{m}I)=\frac{n(m-n)}{2m}$. Then, after minus the center $\frac{n}{m}I$, we see that $P-\frac{n}{m}I$ and $P^{\perp}-\frac{m-n}{m}I=-(P-\frac{n}{m}I)$ are two minimal embeddings of $G(n,m;\mathbb{F})$ and $G(m-n,m;\mathbb{F})$ in the same one hypersphere contained in $H(m;\mathbb{F})$, their images and the associated cones are opposite.

We can reduce the codimension since the images of $G(n,m;\mathbb{F})$ and $G(m-n,m;\mathbb{F})$ in $H(m;\mathbb{F})$ are all located in the affine hyperplane $\{P\in H(m;\mathbb{F}) |trP=0\}$, hence:

\begin{theorem}\label{minimal models}

a). The Grassmannian $G(n,m;\mathbb{F})(\mathbb{F}=\mathbb{R},\mathbb{C},\mathbb{H})$ and $G(m-n,m;\mathbb{F})$ can be embedded into $S^{N-2}(r)$ as an pair of minimal submanifolds simultaneously, where $r=\sqrt{\frac{n(m-n)}{2m}}$, $N=m+dm(m-1)/2$, moreover, their images are opposite;

b). There exists two opposite cones associated with $G(n,m;\mathbb{F})$ in $(N-1)$ dimensional Euclidean space, one is the cone over the embedding: $P-\frac{n}{m}I(P\in G(n,m;\mathbb{F}))$, the other one is the cone over the inverse embedding: $P^{\perp}-\frac{m-n}{m}I$.
\end{theorem}

\subsection{Area-minimizing cones over  general Grassmannians}
The cone and its opposite cone share the same tangent vectors, normal vectors, second fundamental forms and normal radius at the antipodal points, and since the second fundamental forms of projective spaces behave a little difference to those of general Grassmannians, so in this subchapter, we consider cone $C$ over the general Grassmannian manifold $G(n,m;\mathbb{F})(m\geq2n\geq 4)$, i.e. $C$ is not the cones over projective spaces or its opposite cones.

Still denote $w,z\in \{1,i,j,k\}$, note the second fundamental form $h$ is equivariant under $\rho$, we can choose the associated equivariant normal vector fields to reduce the computation to the origin $P_{0}$.

According to \eqref{normal vectors at orgin}, an normal basis of $C$ at $P_{0}$ can be given by: $\xi_{0}=\sqrt{\frac{2}{m}}I_{m}$, $H_{l}=e_{1}e_{1}^{T}-e_{l}e_{l}^{T}(2\leq l\leq n)$, $H_{\gamma}=e_{n+1}e_{n+1}^{T}-e_{\gamma}e_{\gamma}^{T}(n+2\leq \gamma \leq m)$, $H_{cd}^{w}=we_{c}e_{d}^{T}+\bar w e_{d}e_{c}^{T}(1\leq c<d\leq n)$, $H_{\lambda \mu}^{z}=z e_{\lambda}e_{\mu}^{T}+\bar z e_{\mu}e_{\lambda}^{T}(n+1\leq \lambda<\mu\leq m)$, where $T$ denote the transpose. All normals in the above have length $1$ and orthogonal to each other except any two in $H_{l}$ or any two in $H_{\gamma}$. We note here $g(H_{l},H_{b})=\frac{1}{2}$, and $g(H_{\gamma},H_{\alpha})=\frac{1}{2}$, if there exits $2\leq l< b\leq n$ or $n+2\leq \gamma <\alpha \leq m$.

We continue to compute:

$g(h(F_{\alpha a}^{u},F_{\beta b}^{v}),\xi_{0})=0$;

$g(h(F_{\alpha a}^{u},F_{\beta b}^{v}),H_{l})=-\delta_{\alpha \beta}(\delta_{a1}\delta_{b1}-\delta_{al}\delta_{bl})\delta_{uv}$, where in the computation, the term $\mathrm{Re}(\bar u v+\bar v u)/2=\mathrm{Re}(\bar u v)=\mathrm{Re} u \ \mathrm{Re} v+\langle \mathrm{Im} \ u,\mathrm{Im} \ v\rangle_{R} =\delta_{uv}$;

$g(h(F_{\alpha a}^{u},F_{\beta b}^{v}),H_{\gamma})=\delta_{ab}(\delta_{\alpha n+1}\delta_{\beta n+1}-\delta_{\alpha \gamma}\delta_{\beta \gamma})\delta_{uv}$;

$g(h(F_{\alpha a}^{u},F_{\beta b}^{v}),H_{cd}^{w})=-\delta_{\alpha \beta}[\mathrm{Re}(\bar u v w)\delta_{ad}\delta_{bc}+\mathrm{Re}(\bar u v \bar w)\delta_{ac}\delta_{bd}]$;

$g(h(F_{\alpha a}^{u},F_{\beta b}^{v}),H_{\lambda \mu}^{z})=\delta_{ab}[\mathrm{Re}(u \bar v z)\delta_{\alpha \mu}\delta_{\beta \lambda}+\mathrm{Re}(u \bar v \bar z)\delta_{\alpha \lambda}\delta_{\beta \mu}]$.

Choose a unit normal vector at $P_{0}$,
\begin{equation}\notag
\xi=s\xi_{0}+\sum_{l=2}^{n}f_{l}H_{l}+\sum_{\gamma=n+2}^{m}g_{\gamma}H_{\gamma}+\sum_{w,c<d}p_{cd}^{w}H_{cd}^{w}+\sum_{z,\lambda<\mu}q_{\lambda \mu}^{z}H_{\lambda \mu}^{z},
\end{equation}

then
\begin{multline}\label{unit normal}
s^{2}+\frac{1}{2}\sum_{l=2}^{n}f_{l}^{2}+\frac{1}{2}\left(\sum_{l=2}^{n}f_{l}\right)^{2} \\
+ \frac{1}{2}\sum_{\gamma=n+2}^{m}g_{\gamma}^{2} +\frac{1}{2}\left(\sum_{\gamma=n+2}^{m}g_{\gamma}\right)^{2}
+\sum_{w,c<d}(p_{cd}^{w})^{2}+\sum_{z,\lambda<\mu}(q_{\lambda \mu}^{z})^{2}=1.
\end{multline}

We denote $h_{u\alpha a,v\beta b}^{\xi}=g(h(F_{\alpha a}^{u},F_{\beta b}^{v}),\xi)$, and $||h^{\xi}||^{2}=\sum_{u\alpha a,v\beta b}(h_{u\alpha a,v\beta b}^{\xi})^{2}$, then
\begin{equation}\notag
\begin{aligned}
h_{u\alpha a,v\beta b}^{\xi}&=
-\delta_{\alpha \beta}(\delta_{a1}\delta_{b1}-\delta_{al}\delta_{bl})\delta_{uv}f_{l}
+\delta_{ab}(\delta_{\alpha n+1}\delta_{\beta n+1}-\delta_{\alpha \gamma}\delta_{\beta \gamma})\delta_{uv}g_{\gamma} \\
&-\delta_{\alpha \beta}\Big[\mathrm{Re}(\bar u v w)\delta_{ad}\delta_{bc}+\mathrm{Re}(\bar u v \bar w)\delta_{ac}\delta_{bd}\Big]p_{cd}^{w} \\
&+\delta_{ab}\Big[\mathrm{Re}(u \bar v z)\delta_{\alpha \mu}\delta_{\beta \lambda}+\mathrm{Re}(u \bar v \bar z)\delta_{\alpha \lambda}\delta_{\beta \mu}\Big]q_{\lambda \mu}^{z}.
\end{aligned}
\end{equation}

Note $c<d,\lambda<\mu$, all the nonzero terms are divided into eight types:

(1)~$\alpha=\beta,a>b$, the results are: $-\mathrm{Re}(\bar u v w)p_{ba}^{w}$;

(2)~$\alpha=\beta,a<b$, the results are: $-\mathrm{Re}(\bar u v \bar w)p_{ab}^{w}$;

(3)~$\alpha<\beta,a=b$, the results are: $\mathrm{Re}(u \bar v \bar z)q_{\alpha \beta}^{z}$;

(4)~$\alpha>\beta,a=b$, the results are: $\mathrm{Re}(u \bar v z)q_{\beta \alpha}^{z}$;

(5)~$\alpha=\beta=n+1,a=b\geq2$, the results are: $\delta_{uv}(f_{a}+\sum_{\gamma}g_{\gamma})$;

(6)~$\alpha=\beta\geq n+2,a=b=1$, the results are: $\delta_{uv}(-\sum_{l}f_{l}-g_{\alpha})$;

(7)~$\alpha=\beta\geq n+2,a=b\geq2$, the results are: $\delta_{uv}(f_{a}-g_{\alpha})$;

(8)~$\alpha=\beta=n+1,a=b=1$, the results are: $\delta_{uv}(-\sum_{l}f_{l}+\sum_{\gamma}g_{\gamma})$.

\begin{prop}
\begin{equation}\notag
\begin{aligned}
||h^{\xi}||^{2}&=d\Bigg[(m-n)\sum_{l=2}^{n}f_{l}^{2}+n\sum_{\gamma=n+2}^{m}g_{\gamma}^{2}
+(m-n)(\sum_{l=2}^{n}f_{l})^{2}+n(\sum_{\gamma=n+2}^{m}g_{\gamma})^{2}\\
&+2\sum_{w,a<b}(p_{ab}^{w})^{2}
+2\sum_{z,\alpha<\beta}(q_{\alpha  \beta}^{z})^{2}\Bigg].
\end{aligned}
\end{equation}
\end{prop}
\begin{proof}
The cases: $\mathbb{F}=\mathbb{R}$ and $\mathbb{F}=\mathbb{C}$ are ommited. For $\mathbb{F}=\mathbb{H}$, when $u,v$ are fixed, there are only one $w$ or $z$, such that $\mathrm{Re}^{2}(\bar u v w)$,$\mathrm{Re}^{2}(\bar u v \bar w)$, or $\mathrm{Re}^{2}(u \bar v \bar z)$,$\mathrm{Re}^{2}(u \bar v z)$ equals $1$. The number of these nonzero terms are sixteen which are divided into four terms of $w=1$, four terms of $w=i$, four terms of $w=j$, four terms of $w=k$, the same for $z$. the other nonzero terms are the sum of $(5)\sim(8)$, then we get the conclusion.
\end{proof}

The maximum of $||h^{\xi}||^{2}$ can be obtained by the following discussion, from \eqref{unit normal},

\begin{equation}\notag
\begin{aligned}
||h^{\xi}||^{2}&\leq d\Bigg[(m-n-1)\Big(\sum_{l=2}^{n}f_{l}^{2}+(\sum_{l=2}^{n}f_{l})^{2}\Big)\\
&+(n-1)\Big(\sum_{\gamma=n+2}^{m}g_{\gamma}^{2}+(\sum_{\gamma=n+2}^{m}g_{\gamma})^{2}\Big)-2s^2\Bigg]\\
&\leq 2d(m-n).
\end{aligned}
\end{equation}

$||h^{\xi}||^{2}=2d(m-n)$ if and only if $s=0$, and all $g_{\gamma},p_{cd}^{w},q_{\lambda \mu}^{z}$ are zeros.

Hence, we have
\begin{prop}
The upper bound of second fundamental form of cones over $G(n,m;\mathbb{F})$ at the points belong to unit sphere contained in $H(m;\mathbb{F})$ is given by: $sup_{\xi}||h^{\xi}||^{2}=\frac{dn(m-n)^{2}}{m}$, where $m\geq 2n\geq 4$.
\end{prop}
\begin{proof}
Since $G(n,m;\mathbb{F})$ is an embedded minimal submanifold in a sphere of radius $\sqrt{\frac{n(m-n)}{2m}}$, then the result follows by multiply the square of radius, see \eqref{mim submfd and its cone}.
\end{proof}

For the normal radius of $C$, we have the following proposition, the result for $\mathbb{F}=\mathbb{R},\mathbb{C}$ is established in \cite{kerckhove1994isolated}, more generally, the normal radius for the canonical embedding of symmetric $R$-spaces are computed by using the Weyl group(\cite{kanno2002area}).

\begin{prop}\label{radius}
The normal radius of cones over $G(n,m;\mathbb{F})$ and $G(m-n,m;\mathbb{F})(\mathbb{F}=\mathbb{R},\mathbb{C},\mathbb{H})$  is $\arccos(1-\frac{m}{n(m-n)})$, it doesn't depend on the base field $\mathbb{F}$, and this is also right for projective space and Grassmannian of hyperplanes.
\end{prop}
\begin{proof}
As an orbit of isometry group action, we can only compute the normal radius at the origin $E_{0}:=P_{0}-\frac{n}{m}I$. Choose an normal vector $\xi$ of the cone at the origin with the same length of $E_{0}$, assume the normal geodesic $\mathrm{cos}(\theta)E_{0}+\mathrm{sin}(\theta)\xi$ intersects $G(n,m;\mathbb{F})$ at another point $P$, i.e.,
\begin{equation}\notag
\mathrm{cos}(\theta)E_{0}+\mathrm{sin}(\theta)\xi=P-\frac{n}{m}I.
\end{equation}

Let $\xi=\begin{pmatrix}C&0\\ 0&D \end{pmatrix}$, where $C\in H(n;\mathbb{F}),D\in H(m-n;\mathbb{F})$, follow \cite{rodman2014topics}, there exist $\mathbb{F}$-unitary matrices $Q,T$, such that $QCQ^{*}=\Lambda_{1}$,
$TDT^{*}=\Lambda_{2}$ are real diagonal matrices. Let $U=\begin{pmatrix}Q&0\\ 0&T \end{pmatrix}$, then $UE_{0}U^{*}=E_{0}$ and $U\xi U^{*}=\begin{pmatrix}\Lambda_{1}&0\\ 0&\Lambda_{2}\end{pmatrix}$ are all diagonal matrices. Then $U(P-\frac{n}{m}I)U^{*}$ is also diagonal which has diagonal elements $-\frac{n}{m},\frac{m-n}{m}$ of multiplicities $m-n,n$ since $P^2=P$.

Now $E_{0}=diag\{\frac{m-n}{m},\ldots,\frac{m-n}{m},-\frac{n}{m},\ldots,-\frac{n}{m}\}$, then the points which are nearest to $E_{0}$ are those interchange one pair of $\frac{m-n}{m}$ and $-\frac{n}{m}$ in $P_{0}-\frac{n}{m}I$, we get the conclusion.
\end{proof}

Now, we recall Gary R.Lawlor's work in \cite{lawlor1991sufficient}, for the following $ODE$ (see definition 1.1.6 in \cite{lawlor1991sufficient}):
\begin{equation}\label{VN}
\begin{cases}
\frac{dr}{d\theta}=r \sqrt{r^{2k}(\mathrm{cos}\theta)^{2k-2}\mathrm{inf}_{v}\left(det(I-\mathrm{tan}(\theta) h_{ij}^{v})\right)^{2}-1}\\
r(0)=1,
\end{cases}
\end{equation}
where $h_{ij}^{v}$ is the matrix representation of the second fundamental form of an minimal submanifold $M$ in sphere, $v$ is an unit normal, $k$ is the dimension of cone $C=C(M)$, and $r=r(\theta)$ describe a projection curve, the $ODE$ is build at a fixed point $p\in M$.

Denote the real vanishing angle by $\theta_{0}$(see Definition 1.1.7 in \cite{lawlor1991sufficient}), Lawlor use the following estimates.

Let $\theta_{1}(k,\alpha)$ be the estimated vanishing angle function which replacing $\mathrm{inf}_{v}det(I-\mathrm{tan}(\theta) h_{ij}^{v})$ by an smaller positive-valued function
\begin{equation}\notag
F(\alpha,\mathrm{tan}(\theta),k-1)=\left(1-\alpha \mathrm{tan}(\theta)\sqrt{\frac{k-2}{k-1}}\right)\left(1+\frac{\alpha \mathrm{tan}(\theta)}{\sqrt{(k-1)(k-2)}}\right)^{k-2}
\end{equation}
in \eqref{VN}, where the condition $\alpha^2 \mathrm{tan}^2(\theta_{1})\leq\frac{k-1}{k-2}$
should be satisfied, and $F$ is an decreasing function of $\alpha$ when $\mathrm{tan}(\theta),k$ are fixed, it is also decreasing with respect to $k$ when $\alpha,\mathrm{tan}(\theta)$ are fixed.

Let $\theta_{2}(k,\alpha)$ be the estimated vanishing angle function which replacing $\mathrm{inf}_{v}det(I-\mathrm{tan}(\theta) h_{ij}^{v})$ by
\begin{equation}\notag
lim_{k\rightarrow \infty} F(\alpha,\mathrm{tan}(\theta),k-1)=\big(1-\alpha \mathrm{tan}(\theta)\big)e^{\alpha \mathrm{tan}(\theta)}
\end{equation}
in \eqref{VN}, where the condition $\alpha^2 \mathrm{tan}^2(\theta_{2})\leq1$
should be satisfied, and $(1-\alpha \mathrm{tan}(\theta))e^{\alpha \mathrm{tan}(\theta)}$ is also an decreasing function of $\alpha$ when $\mathrm{tan}(\theta)$ are fixed.

The three angles have the following relation:
\begin{equation}\notag
\theta_{0}\leq \theta_{1}(k,\alpha)\leq \theta_{2}(k,\alpha),
\end{equation}
and Lawlor use the angle function $\theta_{1}$ for $dim\ C=\{3,\ldots,11\}$, the angle function $\theta_{2}$ for $dim\ C=12$ to gain "The table".

In the following, we compare the normal radius and vanishing angles of all the general Grassmannian manifolds $G(n,m;\mathbb{F})(4\leq 2n\leq m)$.

The cones having dimension of at most $12$ are:$G(2,4;\mathbb{R})$, $G(2,5;\mathbb{R})$, $G(2,6;\mathbb{R})$, $G(2,7;\mathbb{R})$, $G(3,6;\mathbb{R})$, $G(2,4;\mathbb{C})$. We list the results in the table below:

\begin{center}
\begin{tabular}{|c|c|c|c|c|}
\hline
$4\leq 2n\leq m$ & $\textrm{dim}\ C$ & $sup||h^{\xi}||^{2}$ & $\textrm{vanishing \ angle} \ \theta_{1}$ & $\textrm{normal \ radius}$ \\ \hline
$G(2,4;\mathbb{R})$ & 5 & 2 & $26.97^{\circ}$ & $\frac{\pi}{2}$ \\ \hline
$G(2,5;\mathbb{R})$ & 7 & 3.6 & $16.20^{\circ}\sim 16.44^{\circ}$ & $\arccos(\frac{1}{6})>80^{\circ}$ \\ \hline
$G(2,6;\mathbb{R})$ & 9 & 5.33 & $11.83^{\circ}\sim 12.14^{\circ}$ & $\arccos(\frac{1}{4})>75^{\circ}$ \\ \hline
$G(2,7;\mathbb{R})$ & 11 & 7.143 & $9.41^{\circ}\sim 9.54^{\circ}$ & $\arccos(\frac{3}{10})>70^{\circ}$ \\ \hline
$G(3,6;\mathbb{R})$ & 10 & 4.5 & $10.23^{\circ}$ & $\arccos(\frac{1}{3})>70^{\circ}$ \\ \hline
$G(2,4;\mathbb{C})$ & 9 & 4 & $11.57^{\circ}$ & $\frac{\pi}{2}$ \\ \hline
\end{tabular}
\end{center}
two times of $\theta_{1}$ are still less than the associated normal radius, so follow the Curvature Criterion(A simplified version refers to Theorem 1.3.5 in \cite{lawlor1991sufficient}), these cones are area-minimizing.

When the dimension of the cones $C$ over Grassmannian are greater than $12$, they are:

(1)~$G(2,m;\mathbb{R})(m\geq8)$, $G(3,m;\mathbb{R})(m\geq7)$ and $G(n,m;\mathbb{R})(m\geq2n\geq8)$;

(2)~$G(2,m;\mathbb{C})(m\geq5)$ and $G(n,m;\mathbb{C})(m\geq2n\geq6)$;

(3)~$G(n,m;\mathbb{H})(m\geq2n\geq4)$.

For estimating of vanishing angle, we use the following formula given by Gary R. Lawlor, if $dim(C)=k>12$, then
\begin{equation}\notag
\mathrm{tan}\big(\theta_{2}(k,\alpha)\big)< \frac{12}{k}\mathrm{tan}\left(\theta_{2}(12,\frac{12}{k}\alpha)\right),
\end{equation}

For Grassmannian $G(n,m;\mathbb{F})$, $k=dn(m-n)+1, \alpha=\sqrt{\frac{dn(m-n)^{2}}{m}}$, the normal radius is $\arccos(1-\frac{m}{n(m-n)})$, then
\begin{equation}\notag
\begin{aligned}
\mathrm{tan}\big(\theta_{2}(k,\alpha)\big)&< \frac{12}{dn(m-n)+1}\mathrm{tan}\left(\theta_{2}(12,\frac{12}{dn(m-n)+1}\sqrt{\frac{dn(m-n)^{2}}{m}})\right) \\
&<\frac{12}{dn(m-n)+1}\mathrm{tan}\left(\theta_{2}(12,\frac{12}{\sqrt{dnm}})\right),
\end{aligned}
\end{equation}

(1)~$d=1$, $m\geq 7$, $\frac{12}{\sqrt{dnm}}\leq 3$, $\mathrm{tan}\big(\theta_{2}(12,3)\big)=\mathrm{tan}(8.64^{\circ})\approx 0.152$,
then \begin{equation}\notag
\theta_{2}(k,\alpha)<\arctan(\frac{2}{n(m-n)+1})<\frac{2}{n(m-n)+1}<\frac{1}{m-2}.
\end{equation}

We will show that
\begin{equation}\label{inequality one}
\frac{2}{m-2}<\arccos\left(1-\frac{m}{n(m-n)}\right)
\end{equation}
is always true when $m\geq 7$, i.e. two times of vanishing angle is still less than the normal radius which satisfy the criterion given by Gary R. Lawlor.

The right hand of \eqref{inequality one} is greater than $\arccos(1-\frac{4}{m})$, so it is sufficient if $\cos(\frac{2}{m-2})>1-\frac{4}{m}$, then \eqref{inequality one} is true. Since $\cos(\frac{2}{m-2})>1-\frac{2}{(m-2)^{2}}$ and $2(m-2)^{2}\geq m(m\geq 7)$, then we get our conclusion.
Hence the associated cones are area-minimizing.

(2)~$d=2$, $m\geq 5$, $\frac{12}{\sqrt{dnm}}\leq \frac{6}{\sqrt{5}}$,  $\mathrm{tan}\left(\theta_{2}(12,\frac{6}{\sqrt{5}})\right)<\mathrm{tan}(8.64^{\circ})\approx 0.152$,

then \begin{equation}\notag
\theta_{2}(k,\alpha)<\arctan(\frac{2}{2n(m-n)+1})<\frac{2}{2n(m-n)+1}<\frac{1}{2(m-2)}.
\end{equation}

We will show that
\begin{equation}\label{inequality two}
\frac{1}{m-2}<\arccos\left(1-\frac{m}{n(m-n)}\right)
\end{equation}
is always true when $m\geq 5$.

The right hand of \eqref{inequality two} is greater than $\arccos(1-\frac{4}{m})$, so it is sufficient if $\cos(\frac{1}{m-2})>1-\frac{4}{m}$, then \eqref{inequality two} is true. Since $\cos(\frac{1}{m-2})>1-\frac{1}{2(m-2)^{2}}$ and $8(m-2)^{2}\geq m(m\geq 5)$, then we get our conclusion.
Hence the associated cones are area-minimizing.

(3)~$d=4$, $m\geq 4$, $\frac{12}{\sqrt{dnm}}\leq \frac{3}{\sqrt{2}}$,  $\mathrm{tan}\left(\theta_{2}(12,\frac{3}{\sqrt{2}})\right)=\mathrm{tan}(8.25^{\circ})\approx 0.145$,
then \begin{equation}\notag
\theta_{2}(k,\alpha)<\arctan(\frac{2}{4n(m-n)+1})<\frac{2}{4n(m-n)+1}<\frac{1}{4(m-2)},
\end{equation}

We will show that
\begin{equation}\label{inequality three}
\frac{1}{4(m-2)}<\arccos\left(1-\frac{m}{n(m-n)}\right)
\end{equation}
is always true when $m\geq 4$.

The right hand of \eqref{inequality three} is greater than $\arccos(1-\frac{4}{m})$, so it is sufficient if $\cos(\frac{1}{4(m-2)})>1-\frac{4}{m}$, then \eqref{inequality three} is true. Since $\cos(\frac{1}{4(m-2)})>1-\frac{1}{32(m-2)^{2}}$ and $128(m-2)^{2}\geq m(m\geq 4)$, then we get our conclusion.
Hence the associated cones are area-minimizing.

\begin{theorem}
The cones over nonoriented real Grassmannian
$G(n,m;\mathbb{R})$, complex Grassmannian $G(n,m;\mathbb{C})$, quaternion Grassmannian $G(n,m;\mathbb{H})$ are area-minimizing, where $m\geq 2n \geq 4$.
\end{theorem}

\begin{remark}
$G(n,m;\mathbb{R})$ is oriented if and only if $m$ is even(see\cite{oproiu1977some}), so those area-minimizing cones over unoriented Grassmannians $G(n,m;\mathbb{R})(m$ is odd) generalize the results of Gary R. Lawlor's.
\end{remark}

\section{Cones over projective spaces}
Now we consider projective space: $\mathbb{F}P^{m-1}=G(1,m;\mathbb{F})(m\geq 2,\mathbb{F}=\mathbb{C},\mathbb{H})$, the cones over $\mathbb{R}P^{m-1}$ has been shown area-minimizing in \cite{lawlor1991sufficient} except the case $m=3$. when $m=3$, the cone over embedding of $\mathbb{R}P^{2}$ is the Veronese cone which owns the least area among a large class of comparison surfaces\cite{murdoch1991twisted},\cite{lawlor1995note}, though, we note here the Veronese cone being area-minimizing is still an open problem. In this section, we prove that the cones over $\mathbb{C}P^{m-1}$ and $\mathbb{H}P^{m-1}$ are all area-minimizing for $m\geq 2$.

We will adopt the following ranges of indices in this section unless otherwise stated:
\begin{center}
$2\leq \alpha,\beta,\lambda,\mu \leq m$;

$u,v,z \in \{1, i, j, k\}$,
\end{center}
where $i,j,k$ denote the imaginary units of $\mathbb{H}$.
\subsection{$m=2$}
These cases are trivial, the spheres $\mathbb{F}P^{1}(\mathbb{F}=\mathbb{C},\mathbb{H})$ are embedded into spheres as minimal hypersurfaces, in fact, they are all equators by the reduction of codimension and the associated cones are plane, similar result also holds for Cayley projective line $\mathbb{O}P^{1}$ which is isometry isomorphic to $S^{8}$.

Consider in the ambient space $H(2;\mathbb{F})$ (for $\mathbb{F}=\mathbb{R},\mathbb{C},\mathbb{H},\mathbb{O}$ uniformly), the origin of $\mathbb{F}P^{1}$ is: $P_{0}=E_{11}=e_{1}e_{1}^{T}$, an orthonormal basis of tangent space of $\mathbb{F}P^{1}$ at $P_{0}$(or minus the center) are: $E^{u}_{2}=ue_{2}e_{1}^{T}+\bar u e_{1}e_{2}^{T}$($u=1$ when $\mathbb{F}=\mathbb{R}$, $u\in \{1,i\}$ when $\mathbb{F}=\mathbb{C}$, $u\in \{1,i,j,k\}$ when $\mathbb{F}=\mathbb{H}$, etc). The unit normal to the cone over $\mathbb{F}P^{1}$ at $P_{0}$ is: $\xi_{0}=I_{2}$, the identity matrix.

The second fundamental forms in $H(2;\mathbb{F})$ are given by:
\begin{equation}\notag
h(E^{u}_{2},E^{v}_{2})=\delta_{uv}(-2E_{11}+2E_{22})=-4\delta_{uv}(P_{0}-\frac{I}{2}),
\end{equation}
and when restricted on the sphere, all these second fundamental forms turn to be zero. So after minus the center, the minimal embeddings  $P-\frac{I}{2}$ is totally geodesic. They are all equators by the reduction of codimension and the associated cones are $2$-plane, $3$-plane, $5$-plane and $9$-plane which have no singularity at origin, hence be area-minimizing.

\subsection{$m\geq 3$}
The origin of $\mathbb{F}P^{m-1}(\mathbb{F}=\mathbb{C},\mathbb{H})$ is: $P_{0}=E_{11}=e_{1}e_{1}^{T}$, an orthonormal basis of tangent space of $\mathbb{F}P^{m-1}$ at $P_{0}$(minus the center) are: $E^{u}_{\alpha}=ue_{\alpha}e_{1}^{T}+\bar u e_{1}e_{\alpha}^{T}(2\leq\alpha\leq m)$, an  basis of normal space to the cone over $\mathbb{F}P^{m-1}$ at $P_{0}$ are given by: $\xi_{0}=\sqrt{\frac{2}{m}}I_{m}$, and $H_{l}=e_{2}e_{2}^{T}-e_{l}e_{l}^{T}(3\leq l\leq m)$, $H_{\lambda \mu}^{z}=z e_{\lambda}e_{\mu}^{T}+\bar z e_{\mu}e_{\lambda}^{T}(2\leq\lambda<\mu\leq m)$, where all normals in the above have length $1$ and orthogonal to each other except any two in $H_{l}$, and $g(H_{l},H_{b})=\frac{1}{2}$ if there exits $2\leq l< b\leq m$. We note here when $m=3$, the embeddings of $\mathbb{F}P^{2}(\mathbb{F}=\mathbb{R},\mathbb{C},\mathbb{H})$ are often called Veronese embeddings.

The second fundamental forms are given by:
\begin{equation}\notag
h(E_{\alpha}^{u},E_{\beta}^{v})=-\delta_{\alpha\beta}(\bar uv+\bar vu)E_{11}+(u\bar vE_{\alpha \beta}+v\bar uE_{\beta \alpha}).
\end{equation}

We continue to compute:

$g(h(E_{\alpha}^{u},E_{\beta}^{v}),\xi_{0})=0$;

$g(h(E_{\alpha}^{u},E_{\beta}^{v}),H_{l})=(\delta_{\alpha2}\delta_{\beta2}-\delta_{\alpha l}\delta_{\beta l})\delta_{uv}$;

$g(h(E_{\alpha}^{u},E_{\beta}^{v}),H_{\lambda \mu}^{z})=\mathrm{Re}(u \bar v z)\delta_{\alpha \mu}\delta_{\beta \lambda}+\mathrm{Re}(u \bar v \bar z)\delta_{\alpha \lambda}\delta_{\beta \mu}$.

Choose a unit normal vector at $P_{0}$,
\begin{equation}\notag
\xi=s\xi_{0}+\sum_{l=3}^{m}f_{l}H_{l}+\sum_{z,\lambda<\mu}q_{\lambda \mu}^{z}H_{\lambda \mu}^{z},
\end{equation}

then
\begin{equation}\label{unit normal two}
s^{2}+\frac{1}{2}\sum_{l=3}^{m}f_{l}^{2}+\frac{1}{2}\left(\sum_{l=3}^{m}f_{l}\right)^{2}+\sum_{z,\lambda<\mu}(q_{\lambda \mu}^{z})^{2}=1.
\end{equation}

We denote $h_{u\alpha,v\beta}^{\xi}=g(h(E_{\alpha}^{u},F_{\beta}^{v}),\xi)$, and $||h^{\xi}||^{2}=\sum_{u\alpha,v\beta}(h_{u\alpha,v\beta}^{\xi})^{2}$, then

\begin{equation}\notag
h_{u\alpha,v\beta}^{\xi}=
\sum_{l=3}^{m}(\delta_{\alpha2}\delta_{\beta2}-\delta_{\alpha l}\delta_{\beta l})\delta_{uv}f_{l}
+\Big[\mathrm{Re}(u \bar v z)\delta_{\alpha \mu}\delta_{\beta \lambda}+\mathrm{Re}(u \bar v \bar z)\delta_{\alpha \lambda}\delta_{\beta \mu}\Big]q_{\lambda \mu}^{z}.
\end{equation}

Note $\lambda<\mu$, the nonzero terms are divided into four types:

(1)~$\alpha=\beta=2$, the result is: $\delta_{uv}\sum f_{l}$;

(2)~$\alpha=\beta\geq3$, the results are: $-\delta_{uv}f_{\alpha}$;

(3)~$\alpha>\beta$, the results are: $\mathrm{Re}(u \bar v z)q_{\beta \alpha}^{z}$;

(4)~$\alpha<\beta$, the results are: $\mathrm{Re}(u \bar v \bar z)q_{\alpha \beta}^{z}$.

By \eqref{unit normal two}, we have:
\begin{equation}\notag
||h^{\xi}||^{2}=d\Big[2-2s^{2}-(\sum_{l=3}^{m}f_{l})^{2}\Big]\leq 2d,
\end{equation}
the equality holds if and only if $s=0, \sum_{l=3}^{m}f_{l}=0$.

By multiply the square of radius of sphere $\frac{m-1}{2m}$, we have
\begin{prop}
The upper bound of second fundamental form of cones over $\mathbb{F}P^{m-1}(m\geq 3)$ at the points belong to unit sphere contained in $H(m;\mathbb{F})$ is $sup_{\xi}||h^{\xi}||^{2}=\frac{d(m-1)}{m}$.
\end{prop}

The calculation of normal radius is similar to the cases of general Grassmannians, we conclude that

\begin{prop}
The normal radius of $\mathbb{F}P^{m-1}(m\geq 3,\mathbb{F}=\mathbb{C},\mathbb{H})$ are $\mathrm{arccos}\left(-\frac{1}{m-1}\right)$.
\end{prop}

Now, $k=d(m-1)+1$, $\alpha^{2}=\frac{d(m-1)}{m}$, we will exhibit the estimated vanishing angles for all the cones over $\mathbb{F}P^{m-1}(m\geq 3)$ in the following.

When $\mathbb{F}=\mathbb{R}$, such cases have been studied in \cite{lawlor1991sufficient}.

When $\mathbb{F}=\mathbb{C}$ and  $\mathbb{F}=\mathbb{H}$, we list the cones having dimension no more than $12$ in the table below:

\begin{center}
\begin{tabular}{|c|c|c|c|}
\hline
$\mathbb{F}P^{m-1}$ & $\textrm{dim}\ C$ & $sup||h^{\xi}||^{2}$ & $\textrm{vanishing \ angle} \ \theta_{1}$  \\ \hline
$\mathbb{C}P^{2}$ & 5 & $4/3$ & $23.43^{\circ}\sim 23.73^{\circ}$ \\ \hline
$\mathbb{C}P^{3}$ & 7 & $3/2$ & $14.91^{\circ}$  \\ \hline
$\mathbb{C}P^{4}$ & 9 & $8/5$ & $11.10^{\circ}$  \\ \hline
$\mathbb{C}P^{5}$ & 11 & $5/3$ & $8.87^{\circ}\sim 8.88^{\circ}$  \\ \hline
$\mathbb{H}P^{2}$ & 9 & $8/3$ & $11.26^{\circ}\sim 11.36^{\circ}$  \\ \hline
\end{tabular}
\end{center}
all the above cones have vanishing angle no more than $\frac{\pi}{4}$, so they are all area-minimizing.

When dim\ $C>12$, the associated projective space are:

(1)~$\mathbb{C}P^{m-1}$, $m\geq 7$;

(2)~$\mathbb{H}P^{m-1}$, $m\geq 4$.

For estimating of vanishing angle, we use the following formula given by Gary R.Lawlor,
\begin{equation}\notag
\mathrm{tan}\Big(\theta_{2}(k,\alpha)\Big)< \frac{12}{k}\mathrm{tan}\left(\theta_{2}(12,\frac{12}{k}\alpha)\right).
\end{equation}

Now, $k=d(m-1)+1>12$, $\alpha=\sqrt{\frac{d(m-1)}{m}}$, then
\begin{equation}\notag
\begin{aligned}
\mathrm{tan}\Big(\theta_{2}(k,\alpha)\Big)&< \frac{12}{k}\mathrm{tan}\left(\theta_{2}(12,\frac{12}{d(m-1)+1}\sqrt{\frac{d(m-1)}{m}})\right) \\
&<\frac{12}{k}\mathrm{tan}\left(\theta_{2}(12,\frac{12}{\sqrt{d(m-1)m}})\right) \\
&<\mathrm{tan}\left(\theta_{2}(12,\frac{12}{\sqrt{d(m-1)m}})\right) \\
\end{aligned}
\end{equation}

(1)~$\mathbb{C}P^{m-1}$, $d=2$, $m\geq 7$, then $\frac{12}{\sqrt{d(m-1)m}}\leq \sqrt{\frac{12}{7}}$, by checking "The Table", we have
\begin{equation}\notag
\theta_{2}(k,\alpha)<\theta_{2}\left(12,\sqrt{\frac{12}{7}}\right)<8.07^{\circ}<\frac{\pi}{4}.
\end{equation}

(2)~$\mathbb{H}P^{m-1}$, $d=4$, $m\geq 4$, then $\frac{12}{\sqrt{d(m-1)m}}\leq \sqrt{3}$, by checking "The Table", we have
\begin{equation}\notag
\theta_{2}(k,\alpha)<\theta_{2}(12,\sqrt{3})=8.15^{\circ}<\frac{\pi}{4}.
\end{equation}

Hence, we have
\begin{theorem}
The cones over $\mathbb{C}P^{m-1}$,$\mathbb{H}P^{m-1}$ are area-minimizing, where $m\geq 2$.
\end{theorem}

\section{Cones over Cayley plane}
Similar to the projective spaces, the Cayley plane $\mathbb{O}P^2$ can also be identified as the set of Hermitian orthogonal projectors, then be embedded as an minimal submanifold in a $25$-dimensional sphere contained in the exceptional Jordan algebra $H(3,\mathbb{O})$(\cite{sakamoto1977planar},\cite{chen1984total}), and it is one of the Veronese embeddings of $\mathbb{F}P^2(\mathbb{F}=\mathbb{R},\mathbb{C},\mathbb{H},\mathbb{O})$. In $2015$, Shinji Ohno and Takashi Sakai first confirmed its cone being area-minimizing from the point view of canonical embedding of symmetric $\mathbb{R}$-space(\cite{ohno2015area}), also see an point view of isoparametric theory(\cite{tang2020minimizing}). In this section, we will give an direct proof for it from the point view of Hermitian orthogonal projectors, some basic facts associated to $\mathbb{O}P^2$ are also exhibited. This section is mainly based on \cite{sakamoto1977planar},\cite{harvey1982calibrated}, \cite{harvey1990spinors}.

The octonions, also called the Cayley numbers, are the last algebra in the Cayley-Dickson sequence which form a divison algebra. It is $\mathbb{O}=\mathbb{H}\oplus \mathbb{H}$ with the following multiplication
\begin{equation}
(a,b)(c,d):=(ac-\bar db,da+b\bar c),\notag
\end{equation}
where $(a,b),(c,d)\in\mathbb{H}\oplus \mathbb{H}$, the conjugation operator is defined by
\begin{equation}
\overline{(a,b)}=(\bar a,-b).\notag
\end{equation}

For an octonion $u$, writing $\mathrm{Re}\ u=\frac{u+\bar u}{2}, \mathrm{Im} \ u=\frac{u-\bar u}{2}$, the octonion can be identified with the Euclidean space $\mathbb{R}^8$ with the inner product
\begin{equation}\notag
\langle u,v \rangle=\mathrm{Re}\ (u\bar v),
\end{equation}
and the norm is defined as $|u|=\sqrt{\langle u,u \rangle}$.

Let $H(3,\mathbb{O})$ be the set of all $3\times 3$ hermitian matrices whose entries are octonions, i.e. $H(3,\mathbb{O})=\{A\in M(3,\mathbb{O})|A^{*}=A\}$, it is an Jordan algebra with the Jordan multiplication
\begin{equation}\notag
A\circ B=\frac{AB+BA}{2} \ for \ A,B \in H(3,\mathbb{O}),\notag
\end{equation}
it is also an real Euclidean space of dimension $27$ with the inner product
\begin{equation}\notag
\langle A,B \rangle=\frac{tr\ (A\circ B)}{2} \ for \ A,B \in H(3,\mathbb{O}),
\end{equation}
and norm $|A|:=\sqrt{\langle A,A \rangle}$.

Each element $A\in H(3,\mathbb{O})$ has the typical form
\begin{equation}\notag
A=\begin{pmatrix}
r_{1} & \bar x_{3} &\bar x_{2} \\
x_{3} & r_{2} & x_{1}\\
x_{2} & \bar x_{1} & r_{3}
\end{pmatrix},
\end{equation}
where $r_{1},r_{2},r_{3}\in \mathbb{R}, x_{1},x_{2},x_{3}\in \mathbb{O}$, or simply written, $A=\{\mathbf{r},\mathbf{x}\}=\{\mathbf{r},(x_{1},x_{2},x_{3})\}$, if $B=\{\mathbf{s},\mathbf{y}\}$, then
\begin{equation}\notag
\langle A,B \rangle=\frac{1}{2}\langle \mathbf{r},\mathbf{s} \rangle+\sum_{i=1}^{3}\langle x_{i},y_{i} \rangle.
\end{equation}

\begin{definition}
The Cayley plane is defined by
\begin{equation}\notag
\mathbb{O}P^2\equiv \{A\in H(3,\mathbb{O})|A^2=A,\ tr_{\mathbb{O}}A=1\}.
\end{equation}
\end{definition}

\begin{prop}\cite{harvey1990spinors}
\begin{equation}\notag
\mathbb{O}P^2= \left\{a \bar{a}^{t}|a^{t}=(a_{1},a_{2},a_{3})\in \mathbb{O}^{3}, |a_{1}|^{2}+|a_{2}|^{2}+|a_{3}|^{2}=1, [a_{1},a_{2},a_{3}]=0 \right\},
\end{equation}
which is a $16$-dimensional compact submanifold of $H(3,\mathbb{O})$.
\end{prop}

\begin{remark}
The above theorem tells us, for every $A\in \mathbb{O}P^2$, it can be written as an octonion matrix whose elements belong to one of the distinguished quaternion subalgebra $\mathbb{\widetilde{H}}\subset \mathbb{O}$, so in fact $A\in H(3,\mathbb{\widetilde{H}})$,  the set of all distinguished quaternion subalgebras, the so called associative Grassmannian, is isomorphic to the homogeneous space $G_{2}/SO(4)$\cite{harvey1982calibrated}.
\end{remark}

The exceptional compact lie group $F_{4}$ is defined as the automorphism group of the Jordan algebra $H(3,\mathbb{O})$, for $A,B\in H(3,\mathbb{O})$,
\begin{equation}\notag
F_{4}\equiv \{g\in GL(H(3,\mathbb{O}))|g(A\circ B)=g(A)\circ g(B)\},
\end{equation}
or equivalently,
\begin{equation}\notag
F_{4}\equiv \{g\in GL(H(3,\mathbb{O}))|g(A^2)=(g(A))^2\},
\end{equation}
where $A^2\equiv A\circ A$.

\begin{lemma}\cite{harvey1990spinors}
If $g\in F_{4}$, then
\begin{equation}\notag
tr \ g(A)=tr \ A,
\end{equation}
for all $A\in H(3,\mathbb{O})$.
\end{lemma}

Since $|A|^2=\frac{tr(A^{2})}{2}$,then the action of $F_{4}$ on $H(3,\mathbb{O})$ is an isometric action, $F_{4}$ is an subgroup of $O(27)$, moreover this action is transitive.

\begin{prop}\cite{harvey1990spinors}
$F_{4}$ acts transitively on the Cayley plane $\mathbb{O}P^{2}$ with isotropy subgroup equals to (an isomorphic copy of) $Spin(9)$ at the point
\begin{equation}\notag
E_{1}\equiv
\begin{pmatrix}
1 &0 &0 \\
0 &0 &0 \\
0 &0 &0
\end{pmatrix},
\end{equation}
i.e.,
\begin{equation}\notag
\mathbb{O}P^2\cong F_{4}/Spin(9).
\end{equation}
\end{prop}

We choose a curve $\alpha (t)\in \mathbb{O}P^2$, $\alpha (0)=P, \alpha^{'}(0)=X\in T_{P}\mathbb{O}P^2 \subset H(3,\mathbb{O})$. Then derivative $\alpha (t)\alpha(t)=\alpha (t)$, we obtain $PX+XP=X$(or $P\circ X+X\circ P=X$). Now Assume $P=g(E_{1})$, denote $X=g(X_{0})$, where $g\in F_{4}$, $X_{0}\in T_{E_{1}}\mathbb{O}P^2$, then $g(X_{0})=g(X_{0})\circ P+P\circ g(X_{0})=g(X_{0}\circ E_{1}+E_{1}\circ X_{0})$, tells us $X_{0}=X_{0}\circ E_{1}+E_{1}\circ X_{0}$, by a similar dimension talk like before, we see $T_{P}\mathbb{O}P^2=\{X\in H(3,\mathbb{O})|XP+PX=X\}$, $T_{E_{1}}\mathbb{O}P^2=\{X\in H(3,\mathbb{O})|XE_{1}+E_{1}X=X\}$.

Since $\mathbb{O}P^{2}$ is the orbit through $E_{1}$ under the isometric action of $F_{4}$, the second fundamental forms are all the same at each point, so in the following, we can only consider these at the origin $E_{1}$.

Follow the discussion above, the tangent space of $\mathbb{O}P^{2}$ at $E_{1}$ is:
\begin{equation}\notag
T_{E_{1}}\mathbb{O}P^{2}=\Biggl\{
\begin{pmatrix}
0 &u &v \\
\bar{u} &0 &0\\
\bar{v} &0 &0
\end{pmatrix},
u,v\in \mathbb{O}
\Biggr\}.
\end{equation}

The normal space at $E_{1}$ is:
\begin{equation}\notag
N_{E_{1}}\mathbb{O}P^{2}=\Biggl\{
\begin{pmatrix}
r_{1} &0 &0 \\
0 &r_{2} &z\\
0 &\bar{z} &r_{3}
\end{pmatrix},
r_{1},r_{2},r_{3}\in \mathbb{R},z\in \mathbb{O}
\Biggr\}.
\end{equation}

Note that the exceptional Jordan algebra $H(3,\mathbb{O})$ is not associative, so it seems that there doesn't exists an similar Gauss formula like other Veronese embeddings of  $\mathbb{F}P^2$($\mathbb{F}=\mathbb{R},\mathbb{C},\mathbb{H}$), since we need an well-defined expression $P(t)Y(t)P(t)$ in the proof in proposition \ref{Normal forms}.

Though, as one of the symmetric space of rank $1$, $\mathbb{O}P^2$ stills behave like $\mathbb{F}P^2(\mathbb{F}=\mathbb{R},\mathbb{C},\mathbb{H})$ in some ways. Kunio Sakamoto\cite{sakamoto1977planar} proved that the Veronese embedding of $\mathbb{O}P^2$ is also an planar geodesic map just like the cases $\mathbb{F}P^2$, hence be an isotropy embedding and the second fundamental form is parallel. In the next, we compute the second fundamental forms of $\mathbb{O}P^{2}$ and the normal radius directly.

$\mathbb{O}P^2$ is covered by three charts(\cite{harvey1990spinors}),
\begin{equation}\notag
\begin{aligned}
U_{1}&\equiv \left \{\frac{a\bar{a}^{t}}{|a|^2}|a^t=(1,x,y)\in \mathbb{O}^3\right \}\cong \mathbb{O}^2,\\ \notag
U_{2}&\equiv \left \{\frac{a\bar{a}^{t}}{|a|^2}|a^t=(x,1,y)\in \mathbb{O}^3\right \}\cong \mathbb{O}^2,\\ \notag
U_{3}&\equiv \left \{\frac{a\bar{a}^{t}}{|a|^2}|a^t=(x,y,1)\in \mathbb{O}^3\right \}\cong \mathbb{O}^2,
\end{aligned}
\end{equation}
we will do the calculation on $U_{1}\cong \mathbb{O}^2$.

Set $x=\sum_{i=1}^{8}x_{i}u_{i}$ and $y=\sum_{j=1}^{8}y_{j}u_{j}$, where
\begin{equation}\notag
(u_{1},u_{2},u_{3},u_{4},u_{5},u_{6},u_{7},u_{8})=(1,i,j,k,e,ie,je,ke)
\end{equation}
are the standard basis of $\mathbb{O}$ when identified with the Euclidean space $\mathbb{R}^{8}$.

Restricted on $U_{1}$, the embedding of $\mathbb{O}P^2$ in $H(3,\mathbb{O})$ can be given by:
\begin{equation}\notag
\begin{aligned}
\varphi: \mathbb{O}P^2\supset U_{1} &\rightarrow H(3,\mathbb{O})\\
(x,y) &\mapsto \frac{a\bar{a}^{t}}{|a|^2}-\frac{I}{3},
\end{aligned}
\end{equation}
where $a^{t}=(1,x,y)$, and $E_{1}-\frac{I}{3}$ is corresponding to the coordinate origin $(0,0)$, the image lies in $S^{25}(\frac{1}{\sqrt{3}})$ linearly full(\cite{sakamoto1977planar}).

Set $\bar{i}=i+8$, denote $\varphi_{i}=\frac{\partial \varphi}{\partial x_{i}},\varphi_{\bar{j}}=\frac{\partial \varphi}{\partial y_{j}}$, where $1\leq i,j\leq 8$.

Then $e_{i}:=\varphi_{i}(0,0)=u_{i}E_{21}+\bar{u_{i}}E_{12}$, $e_{\bar{j}}:=\varphi_{\bar{j}}(0,0)=u_{j}E_{31}+\bar{u_{j}}E_{13}(1\leq i,j\leq 8)$  form the orthonormal basis of $T_{E_{1}}\mathbb{O}P^2$.

We continue to compute:
\begin{equation}\notag
\begin{aligned}
\nabla_{e_{i}}e_{j}&=\varphi_{ij}(0,0)=2\delta_{ij}(-E_{11}+E_{22}), \\
\nabla_{e_{\bar{i}}}e_{\bar{j}}&=\varphi_{\bar{i}\bar{j}}(0,0)=2\delta_{ij}(-E_{11}+E_{33}), \\
\nabla_{e_{i}}e_{\bar {j}}&=\varphi_{i \bar {j}}(0,0)=\delta_{ij}(E_{23}+E_{32}),
\end{aligned}
\end{equation}
where $\nabla$ is the Euclidean connection of $H(3,\mathbb{O})$, then $\sum_{i=1}^{8}\nabla_{e_{i}}e_{i}+\sum_{j=1}^{8}\nabla_{e_{\bar{j}}}e_{\bar{j}}$ is parallel to the position vector $E_{1}-\frac{I}{3}$, so its component in the tangent space of sphere is zero, moreover, since $\mathbb{O}P^2$ is an orbit of an isometry group action,  we conclude

\begin{prop}
The embedding $\varphi$ is minimal.
\end{prop}

The normal space $N_{E_{1}}C$ of the cone $C$ over $\mathbb{O}P^2$ at $E_{1}$(or minus the center) has an orthonormal basis given by:
\begin{equation}\notag
L=\sqrt{\frac{2}{3}}I, M=E_{22}-E_{33}, N_{k}=u_{k}E_{23}+\bar{u_{k}}E_{32}(1\leq k\leq 8).
\end{equation}

Choose an unit normal vector $\xi\in N_{E_{1}}C$, $\xi=aL+bM+\sum_{k}c_{k}N_{k}$, i.e.,
\begin{equation}\notag
a^2+b^2+\sum_{k=1}^{8}c_{k}^2=1.
\end{equation}

Set $H_{AB}^{\xi}=\langle H(e_{A},e_{B}),\xi \rangle$, where $H(e_{A},e_{B})$ is the second fundamental form of the embedding $\varphi$ into sphere, $1\leq A,B \leq 16$,
then \begin{equation}\notag
H_{AB}^{\xi}=\langle \varphi_{AB},\xi \rangle=a\langle \varphi_{AB},L \rangle+b\langle \varphi_{AB},M \rangle+\sum_{k}c_{k}\langle \varphi_{AB},N_{k} \rangle,
\end{equation}
since $\xi$ is perpendicular to the position vector(or refer to \cite{lawson1970complete}).

The values of $H_{AB}^{\xi}$ is given by:
\begin{equation}\notag
H_{ij}^{\xi}=b\delta_{ij}, H_{i\bar {j}}^{\xi}=c_{1}\delta_{ij}, H_{\bar{i}j}^{\xi}=c_{1}\delta_{ij}, H_{\bar{i}\bar{j}}^{\xi}=-b\delta_{ij}.
\end{equation}

Hence, $||H^{\xi}||^2:=\sum_{AB}(H_{AB}^{\xi})^2=16(b^2+c_{1}^2)\leq 16$, the equal sign is hold if and only if $a=0$ and $c_{\alpha}=0$ for every $\alpha \in \{2,\ldots,8\}$, i.e.,  $||H^{\xi}||^2$ attains its maximum at the normal direction
\begin{equation}\notag
\begin{pmatrix}
0 &0 &0\\
0 &b &c\\
0 &c &-b
\end{pmatrix},
\end{equation}
where $b,c\in \mathbb{R}$, and $b^2+c^2=1$.

\begin{prop}
The upper bound of second fundamental form of cone over $\mathbb{O}P^2$ at the points belong to unit sphere contained in $H(3,\mathbb{O})$ is given by: $sup_{\xi}||H^{\xi}||^2=\frac{16}{3}$.
\end{prop}

Following \cite{lawlor1991sufficient}, now $dim\ C=17$, $\alpha^2=\frac{16}{3}$, denote the estimated vanishing angle by $\theta_{2}$, then we have
\begin{equation}\notag
\mathrm{tan}\left(\theta_{2}\left(17,\sqrt{\frac{16}{3}}\right)\right)<\frac{12}{17}\mathrm{tan}\left(
\theta_{2}\left(12,\frac{12}{17}\sqrt{\frac{16}{3}}\right)\right).
\end{equation}

Since $\left(\frac{12}{17}\sqrt{\frac{16}{3}}\right)^2<3$, we have $\theta_{2}\left(17,\sqrt{\frac{16}{3}}\right)<\theta_{2}(12,\sqrt{3})=8.15^{\circ}$ by Gary R.Lawlor's table.

Now, we consider the normal radius,
\begin{prop}
The normal radius of the cone over $\mathbb{O}P^2$ is $\frac{2\pi}{3}$.
\end{prop}
\begin{proof}
We can only do the computation at the origin $E_{1}-\frac{I}{3}$, choose an normal vector of the cone $C$ of length $\frac{1}{\sqrt{3}}$: $\xi=\begin{pmatrix}r &0 &0\\0 &s &z \\0 &\bar{z} &t\end{pmatrix}$, where $r,s,t\in \mathbb{R}, z\in \mathbb{O}$, i.e.,
\begin{equation}\notag
r^2+s^2+t^2+2|z|^2=\frac{2}{3},
\end{equation}

and \begin{equation}\notag
2r=s+t.
\end{equation}

Set \begin{equation}\label{tower4}
\mathrm{cos}(\theta)(E_{1}-\frac{I}{3})+\mathrm{sin}(\theta)\xi=P-\frac{I}{3},
\end{equation}
where $\theta\in (0,\pi], P\in \mathbb{O}P^2$.

Since trace of the right hand of \eqref{tower4} is zero, then
\begin{equation}\notag
r+s+t=0.
\end{equation}

Hence
\begin{equation}\notag
r=0,t=-s,s^2+|z|^2=\frac{1}{3}.
\end{equation}

Now, \begin{equation}\notag
P=\begin{pmatrix}
\frac{2}{3}\mathrm{cos}(\theta)+\frac{1}{3} &0 &0\\0 &-\frac{1}{3}\mathrm{cos}(\theta)+\mathrm{sin}(\theta)s+\frac{1}{3} &\mathrm{sin}(\theta)z \\0 &\mathrm{sin}(\theta)\bar{z} &-\frac{1}{3}\mathrm{cos}(\theta)-\mathrm{sin}(\theta)s+\frac{1}{3}
\end{pmatrix}
\end{equation}

From the condition $P^2=P$, we get the solutions of \eqref{tower4}:

(1)~$\theta=\frac{2\pi}{3}$, the set of nearest points to $E_{1}-\frac{I}{3}$ is isomorphic to an immersed $S^{8}$ in $\mathbb{O}P^2$,
\begin{equation}\notag
P=\begin{pmatrix}
0 &0 &0\\
0 &\frac{1+\sqrt{3}s}{2} &\frac{\sqrt{3}z}{2}\\
0 &\frac{\sqrt{3}\bar{z}}{2} &\frac{1-\sqrt{3}s}{2}
\end{pmatrix},
\end{equation}
where $s^2+|z|^2=\frac{1}{3}$.

(2)~$\theta=\pi$.

So, the normal radius is $\frac{2\pi}{3}$.

\end{proof}

Finally,
\begin{theorem}
The cone over $\mathbb{O}P^2$ under embedding $\varphi$ is area-minimizing.
\end{theorem}

\begin{remark}
The four cones over Veronese embeddings of $\mathbb{F}P^2(\mathbb{F}=\mathbb{R},\mathbb{C},\mathbb{H},\mathbb{O})$ own the same normal radius $\frac{2\pi}{3}$, and proportional length of square of second fundamental forms.
\end{remark}
\section{Cones over oriented real Grassmannians}
\subsection{Pl\"{u}cker  embedding of oriented real Grassmannian}
For searching cones over oriented real Grassmannian $\widetilde{G}(n,m;\mathbb{R})$, we should embed them in an suitable ambient Euclidean space.

In this section, we consider the Pl\"{u}cker  embedding  of all oriented real Grassmannian $\widetilde{G}(n,m;\mathbb{R})$ into unit spheres of exterior vector spaces as minimal submanifolds(\cite{harvey1982calibrated},\cite{morgan1985exterior}), it includes the standard embedding of  complex hyperquadric $Q_{l}(\mathbb{C})\cong \widetilde{G}(2,l+2;\mathbb{R})$ into Euclidean space $\wedge^{2}\mathbb{R}^{l+2}$, the calculations of second fundamental forms are based on \cite{chen1988the} which use the method of moving frame, some good references on moving frame are \cite{Chern1983},\cite{griffiths1974cartan}.

Let $V$ be an $m$-dimensional real inner product space(often refer to $\mathbb{R}^{m}$), we can define inner product spaces $\wedge^{n}V$. If $e_{1},\ldots,e_{n}$ are linearly independent in $V$, then the product
\begin{equation}\notag
e_{\lambda}=e_{\lambda_{1}}\wedge\cdots \wedge e_{\lambda_{n}}
\end{equation}
corresponding to all $\lambda \in \mathcal{\wedge}(n,m)$ are linearly independent in $\wedge^{n}V$, where $\mathcal{\wedge}(n,m)$ denote the set of all increasing maps of $\{1,\ldots,n\}$ into $\{1,\ldots,m\}$.

With respect to the induced inner product:
\begin{equation}\notag
\langle e_{\lambda},e_{\mu}\rangle=det(\langle e_{\lambda_{i}},e_{\mu_{j}}\rangle),
\end{equation}
if $\{e_{1},\ldots,e_{m}\}$ are orthonormal basis of $V$, then
$\wedge^{n}V$ have orthonormal basis: $\{e_{\lambda}\}$($\lambda\in \mathcal{\wedge}(n,m)$), and $dim  \wedge^{n}V =C^n_{m}$, where $C_{m}^{n}$ is the combination number.

Let $\wedge^{n}\mathbb{R}^{m}$ be the vector space of all $n$-vectors of $\mathbb{R}^{m}$, for an oriented $n$-plane $L\in\widetilde{G}(n,m;\mathbb{R})$, Let $\{u_{1},\ldots, u_{n}\}$ be an oriented orthonormal basis of $L$, the Pl\"{u}cker  embedding  is given by:
\begin{equation}\notag
\begin{aligned}
i: \widetilde{G}(n,m;\mathbb{R}) &\rightarrow \wedge^{n}\mathbb{R}^{m} \\
L &\mapsto u_{1}\wedge \cdots \wedge u_{n},
\end{aligned}
\end{equation}
the image can be seen as an orbit of the exterior power of standard representation of $SO(m)$ on $\mathbb{R}^{m}$, so it is equivariant.

Choose an oriented orthonormal basis of $\mathbb{R}^{m}$: $E=(e_{1},\ldots,e_{m})$, such that $E_{0}=e_{1}\wedge\cdots \wedge e_{n}$ is the origin of $\widetilde{G}(n,m;\mathbb{R})$, we give an oriented orthonormal basis for $\wedge^{n}\mathbb{R}^{m}$ as follows: set $E_{i_{1}\ldots i_{q} \alpha_{1}\ldots \alpha_{q}}=e_{1}\wedge \cdots \wedge e_{\alpha_{1}}\wedge \cdots \wedge e_{\alpha_{q}}\wedge \cdots \wedge e_{n}$, where $1\leq q \leq min(n,m-n)$, $1\leq i_{1}<\cdots <i_{q} \leq n$, $n+1\leq \alpha_{1}<\cdots <\alpha_{q}\leq m$, and $e_{\alpha_{1}}$ is in the $i_{1}$-position, \ldots, $e_{\alpha_{q}}$ is in the $i_{q}$-position. Then, all of the $E_{0},E_{i_{1}\ldots i_{q} \alpha_{1}\ldots \alpha_{q}}$ gives an oriented orthonormal basis of $\wedge^{n}\mathbb{R}^{m}$ under the lexicographic arrangement.

Denote the Maure-Cartan forms of $SO(m)$ by $\omega$, then $\omega=E^{-1}dE$, where $E\in SO(m)$, i.e. $dE_{A}=E_{B}\omega_{A}^{B}$, it satisfy the Maurer-Cartan equation: $d\omega=-\omega\wedge \omega$.

Now
\begin{equation}\notag
de_{i}=e_{j}\omega_{i}^{j}+e_{\alpha}\omega_{i}^{\alpha},
\end{equation}
where $1\leq i,j\leq n, n+1\leq \alpha \leq m$.

then
\begin{equation}\notag
\begin{aligned}
d(e_{1}\wedge\cdots \wedge e_{n})&=\sum_{i=1}^{n}e_{1}\wedge\cdots \wedge de_{i}\wedge \cdots \wedge e_{n}\\&=\sum_{i,\alpha}e_{1}\wedge\cdots \wedge e_{i-1}\wedge e_{\alpha}\wedge e_{i+1}\wedge  \cdots \wedge e_{n}\omega_{i}^{\alpha}\\
&=E_{i\alpha}\omega_{i}^{\alpha},
\end{aligned}
\end{equation}

$\widetilde{G}(n,m;\mathbb{R})$ is equipped with the induced metric: $ds^{2}=\sum_{i,\alpha}(\omega_{i}^{\alpha})^{2}$, and $\{E_{i\alpha}\}$ is the orthonormal tangent frame.

The orthonormal normal frame of $i(\widetilde{G}(n,m;\mathbb{R}))\hookrightarrow S^{C_{m}^{n}-1}(1)$ is

$\{E_{i_{1}\ldots i_{q} \alpha_{1}\ldots \alpha_{q}}(q\geq 2)\}$, $E_{0}$ is the position vector.

\subsection{Second fundamental forms of $\widetilde{G}(n,m;\mathbb{R})$}
We arrange the following indices: $1\leq i_{1}<\cdots<i_{q}\leq n,\ n+1\leq \alpha_{1}<\cdots<\alpha_{q}\leq m$ where $1\leq q \leq min(n,m-n)$.

Now, choose an element $\sigma\in S_{q}$, where $S_{q}$ is the permutation group of order $q$, let $\tau=\sigma^{-1}$, then easy to see
\begin{equation}\notag
E_{i_{\sigma(1)}\ldots i_{\sigma(q)} \alpha_{1}\ldots \alpha_{q}}=E_{i_{1}\ldots i_{q} \alpha_{\tau(1)}\ldots \alpha_{\tau(q)}}
\end{equation}
where $\alpha_{\tau(1)}$ is in the $i_{1}$-position,$\ldots$,$\alpha_{\tau(q)}$ is in the $i_{q}$-position.

Hence \begin{equation}\label{summ}
E_{j_{1}\cdots j_{q}\beta_{1}\cdots \beta_{q}}=\delta_{j_{1}\cdots j_{q}}^{i_{1}\cdots i_{q}}\delta_{\beta_{1}\cdots \beta_{q}}^{\alpha_{1}\cdots \alpha_{q}}E_{i_{1}\ldots i_{q} \alpha_{1}\ldots \alpha_{q}},
\end{equation}
where $i_{1},\ldots,i_{q},\alpha_{1},\ldots,\alpha_{q}$ are not indices for summing, and \begin{equation}\notag
\langle E_{i_{1}\ldots i_{q} \alpha_{1}\ldots \alpha_{q}},E_{j_{1}\cdots j_{q}\beta_{1}\cdots \beta_{q}}\rangle=\delta_{j_{1}\cdots j_{q}}^{i_{1}\cdots i_{q}}\delta_{\beta_{1}\cdots \beta_{q}}^{\alpha_{1}\cdots \alpha_{q}}.
\end{equation}

Now,
\begin{equation}\notag
\begin{aligned}
dE_{i\alpha}&=\sum_{j\neq i}e_{1}\wedge \cdots \wedge de_{j} \wedge \cdots \wedge e_{\alpha} \wedge \cdots \wedge e_{n} + e_{1}\wedge \cdots \wedge de_{\alpha} \wedge \cdots \wedge e_{n} \\
&=-E_{0}\omega_{i}^{\alpha}+\sum_{j,\beta}E_{j\beta}(\omega_{i}^{j}\delta_{\alpha}^{\beta}
+\omega_{\alpha}^{\beta}\delta_{i}^{j})+\sum_{j\neq i,\beta \neq \alpha}E_{ji\beta \alpha}\omega_{j}^{\beta},
\end{aligned}
\end{equation}
where $e_{\alpha},de_{\alpha}$ are in the $i$-position.

In summary, the structure equations can be written as:
\begin{equation}\notag
\begin{cases}
dE_{0}=E_{i\alpha}\omega_{i}^{\alpha},\\
dE_{i\alpha}=-E_{0}\omega_{i}^{\alpha}+\sum_{j,\beta}E_{j\beta}(\omega_{i}^{j}\delta_{\alpha}^{\beta}
+\omega_{\alpha}^{\beta}\delta_{i}^{j})+\sum_{j\neq i,\beta\neq \alpha}E_{ji\beta \alpha}\omega_{j}^{\beta},\\
dE_{i_{1}\ldots i_{q} \alpha_{1}\ldots \alpha_{q}}=high\ order\ terms.
\end{cases}
\end{equation}

The dual $1$-forms are: $\theta^{(i\alpha)}=(E_{i\alpha})^{*}=\omega_{i}^{\alpha}$, and the connection $1$-forms are:
\begin{equation}\notag
\theta_{(i\alpha)}^{(j\beta)}=\langle dE_{i\alpha},E_{j\beta}\rangle=\omega_{i}^{j}\delta_{\alpha}^{\beta}
+\omega_{\alpha}^{\beta}\delta_{i}^{j}.
\end{equation}

Then $\theta_{(i\alpha)}^{0}=\langle dE_{i\alpha},E_{0}\rangle=-\omega_{i}^{\alpha}=-\theta^{(i\alpha)}$,

and follow \eqref{summ}, we have \begin{equation}\notag
\theta_{(i\alpha)}^{(jk\beta\gamma)}=\langle dE_{i\alpha},E_{jk\beta\gamma}\rangle=\sum_{l,\tau}\delta_{il}^{jk}\delta_{\alpha \tau}^{\beta \gamma}\theta^{(l\tau)}.
\end{equation}

When $q>2$, $\theta_{(i\alpha)}^{(i_{1}\ldots i_{q}\alpha_{1}\ldots \alpha_{q})}=0$.

Hence the coefficients of second fundamental forms are given by:
\begin{equation}\notag
\begin{cases}
h_{(i\alpha)(j\beta)}^{0}=-\delta_{j}^{i}\delta_{\beta}^{\alpha},\\
h_{(i\alpha)(l\tau)}^{(jk\beta\gamma)}=\delta_{il}^{jk}\delta_{\alpha\tau}^{\beta \gamma},\\
h_{(i\alpha)(l\tau)}^{(i_{1}\ldots i_{q}\alpha_{1}\ldots \alpha_{q})}=0(q>2).
\end{cases}
\end{equation}

So, the Pl\"{u}cker embedding is minimal, and the the normals to the cone $C$ of $i(\widetilde{G}(n,m;\mathbb{R}))$ in $\mathbb{R}^{C_{m}^{n}}$ at $E_{0}$ are: $E_{jk\beta\gamma}$ and $E_{i_{1}\ldots i_{q}\alpha_{1}\ldots \alpha_{q}}(q>2)$, so
\begin{equation}\notag
sup_{\xi}||h^{\xi}||^{2}=\sum_{i,\alpha,l,\tau}||h_{(i\alpha)(l \tau)}^{jk\beta\gamma}||^{2}
=\sum_{i,l,\alpha,\tau}(\delta_{il}^{jk})^{2}(\delta_{\alpha\tau}^{\beta\gamma})^{2}=4,
\end{equation}
where $\xi$ is a unit normal of $C$ at $E_{0}$.

Hence,
\begin{theorem}
The upper bound of second fundamental form of cones over $\widetilde{G}(n,m;\mathbb{R})$ contained in $\wedge^{n}\mathbb{R}^{m}$ is $sup_{\xi}||h^{\xi}||^{2}=4$.
\end{theorem}

\subsection{Normal radius of the cones over $\widetilde{G}(n,m;\mathbb{R})$}
In this subsection, we will show that the normal radius of the cones over $\widetilde{G}(n,m;\mathbb{R})$ are all at least $\frac{\pi}{2}$, in fact, they are all equal $\frac{\pi}{2}$. We note here Takahiro Kanno have confirmed it for $\widetilde{G}(2,m;\mathbb{R})$ by computing with Weyl group of the symmetric pair $(SO(m)^2,SO(m))$(\cite{kanno2002area}).

\begin{theorem}
The normal radius of the cones over $i(\widetilde{G}(n,m;\mathbb{R}))$ are all at least $\frac{\pi}{2}$.
\end{theorem}
\begin{proof}
Let $P_{0}=e_{1}\wedge \cdots \wedge e_{n}$ be the origin, $T$ be a normal of $C$ at $P_{0}$, since $i$ is homogeneous, we can do the computation at the origin $P_{0}$.

First, we assume the normal radius is less than $\frac{\pi}{2}$, then there exists a normal $T$, such that
\begin{equation}\notag
P_{0}+T=\lambda P,
\end{equation}
where $\lambda>1$ is a positive real number, $P$ is another point in $i(\widetilde{G}(n,m;\mathbb{R}))$, i.e. $P$ is a unit simple $n$-vector.

If $n=2$, the proof is clear than the general cases. Now $P_{0}=e_{1}\wedge e_{2}$, the tangent vectors are: $\{e_{1}\wedge e_{\alpha},e_{2}\wedge e_{\alpha} \}(3\leq\alpha\leq m)$, the normal vectors are: $\{e_{\alpha}\wedge e_{\beta} \}(3\leq\alpha<\beta\leq m)$.

We set $T=\sum_{\alpha<\beta}a_{\alpha \beta}e_{\alpha}\wedge e_{\beta}$, then $P$ does be exist if and only if $P_{0}+T$
is decomposable, i.e. it is a simple $n$-vector. By the condition about decomposable of $2$-vectors(\cite{morgan1985exterior},\cite{pavan2017exterior}), it is equivalent to:
\begin{equation}\notag
\begin{aligned}
0&=\Big(e_{1}\wedge e_{2}+\sum_{\alpha<\beta}a_{\alpha \beta}e_{\alpha}\wedge e_{\beta}\Big)\wedge\Big(e_{1}\wedge e_{2}+\sum_{\alpha<\beta}a_{\alpha \beta}e_{\alpha}\wedge e_{\beta}\Big)\\
&=2\sum_{\alpha<\beta}e_{1}\wedge e_{2}a_{\alpha \beta}e_{\alpha}\wedge e_{\beta}+\sum_{\alpha<\beta,\gamma<\tau}a_{\alpha \beta}a_{\gamma \tau}e_{\alpha}\wedge e_{\beta}\wedge e_{\gamma}\wedge e_{\tau}
\end{aligned}
\end{equation}
then all $a_{\alpha\beta}$ should be zero, this is a contraction.

For general Grassmanians, we set $T=\sum_{q,I,\alpha}a_{i_{1}\cdots i_{q}\alpha_{1}\cdots \alpha_{q}}E_{i_{1}\cdots i_{q}\alpha_{1}\cdots \alpha_{q}}$, where the indices $I=\{(i_{1},\ldots,i_{q})|1\leq i_{1}<\cdots<i_{q}\leq n\}$, $\alpha=\{(\alpha_{1},\ldots,\alpha_{q})|n+1\leq \alpha_{1}<\cdots<\alpha_{q}\leq m \}$ and
$2\leq q \leq min(n,m-n)$. Denote $P_{0}=e_{1}\wedge\cdots\wedge e_{n}$, set $w=P_{0}+T$, we define a linear operator $T_{w}:\mathbb{R}^{m}\rightarrow \wedge^{n+1}\mathbb{R}^{m}$ by $T_{w}(v)=v\wedge w$, $U:=ker \ T_{w}$. Then $w$ is decomposable if and only if $dim \ U=n$(\cite{federer1969geometric}).

We choose a nonzero vector $v\in U, v=\sum_{A=1}^{m}x_{A}e_{A}$, then
\begin{equation}\notag
\sum_{A=1}^{m}x_{A}e_{A}\wedge\Big(e_{1}\wedge\cdots \wedge e_{n}+\sum_{q,I,\alpha}a_{i_{1}\cdots i_{q}\alpha_{1}\cdots \alpha_{q}}E_{i_{1}\cdots i_{q}\alpha_{1}\cdots \alpha_{q}}\Big)=0.
\end{equation}

We call the axis $(n+1)$-vectors $e_{s_{1}\cdots s_{p}r_{1}\cdots r_{q}}:=e_{s_{1}}\wedge \cdots \wedge e_{s_{p}}\wedge e_{r_{1}}\wedge \cdots \wedge e_{r_{q}}$ is of type $(p,q)$ if $s_{1},\ldots,s_{p}\in \{1,\ldots,n\}$ and $r_{1},\ldots,r_{p}\in \{n+1,\ldots,m\}$, then the terms of type $(n,1)$ are: $\sum_{\alpha}(-1)^{n}x_{\alpha}e_{1}\wedge\cdots \wedge e_{n}\wedge e_{\alpha}$. So, $x_{\alpha}\equiv 0$ for all $\alpha=n+1,\ldots,m$.

Then $dim\ U=n$ is equivalent to $U=span_{\mathbb{R}}\{e_{1},\ldots,e_{n}\}$, i.e.
\begin{equation}\notag
e_{i}\wedge T=0,
\end{equation}
for all $i\in \{1,\ldots,n\}$.

Set $I=\{i_{1},\ldots,i_{q}\}$, $e_{i}\wedge T=0$ is equivalent to that the coefficients $a_{i_{1}\cdots i_{q}\alpha_{1}\cdots \alpha_{q}}$ are all zeros for those indices $i\in I$. But the value of $i,i_{1},\ldots,i_{q}$ could be anyone of $1,\ldots,n$, so all the coefficients $a_{i_{1}\cdots i_{q}\alpha_{1}\cdots \alpha_{q}}$ are zeros, i.e. $T\equiv 0$, this is a contradiction.
\end{proof}

\begin{remark}
 Assume the unit speed normal geodesic $\mathrm{cos}(\theta)P_{0}+\mathrm{sin}(\theta)\frac{T}{|T|}$ interest the image of $i(\widetilde{G}(n,m;\mathbb{R}))$ at another point $P$ where $\theta \in (0,\pi]$, if $\theta \neq \frac{\pi}{2},\pi$, then easy to see the above discussion also derive the same contradiction, so the normal radius is just $\frac{\pi}{2}$, the nearest points are those axis normal vectors and non-axis normal vectors $T$ where $T$ is decomposable too, in special, if $2n\leq m$, it contains the grassmannian of oriented $n$-plane in normal space $\mathbb{R}^{m-n}$.
\end{remark}

By checking the estimated vanishing angle(we omit details here, it is similar to $G(n,m;\mathbb{R})$), we conclude that:
\begin{theorem}
Except $\widetilde{G}(2,4;\mathbb{R})$, all the cones over Pl\"{u}cker  embedding  of $\widetilde{G}(n,m;\mathbb{R})$ are area-minimizing.
\end{theorem}

\begin{remark}
$\widetilde{G}(2,4;\mathbb{R})$ is isometry isomorphic to $S^{2}(\frac{\sqrt{2}}{2})\times S^{2}(\frac{\sqrt{2}}{2})$
(\cite{morgan1985exterior}), the images under Pl\"{u}cker  embedding  is the Clifford minimal hypersurfaces in $S^{5}(1)$, its cone belong to the type of cones which the Curvature Criterion of Lawlor is necessary and sufficient, follow Corollary 4.4.6 in \cite{lawlor1991sufficient}, it is unstable. We note here an pictured description in \cite{morgan2016geometric} for area-minimizing surface bounded by the product of spheres collapsed onto the cone with the increasing dimension.
\end{remark}

\begin{remark}
We can also talk about the cones over Pl\"{u}cker  embedding of $G(n,m;\mathbb{C})$ and $G(n,m;\mathbb{H})$ in a similar way. Moreover, notice these submanifolds are written in $n$-vectors(respectively, $2n,4n$-vectors), consider their canonical forms (\cite{harvey1982calibrated},\cite{harvey1990spinors}), we have the following question: could we find some calibrations which calibrate these area-minimizing cones?
\end{remark}

\vspace{0.3cm}
\noindent\textbf{Acknowledgments}. This work is supported by NSFC No.11871450.

\end{document}